\documentclass[a4paper,reqno,11pt]{amsart} 

\usepackage[T1]{fontenc}
\usepackage[utf8]{inputenc}
\usepackage{xspace,
		amsfonts}
\usepackage{amsmath} 
\usepackage{amssymb}
\usepackage{graphicx,
		bbm,
		tikz,
		subfig}
\usepackage{amsthm}
\usepackage{mathtools}
\usepackage{enumitem}
\usepackage{verbatim}
\usepackage[autostyle]{csquotes}

\usepackage{cite}

\usepackage{geometry}
\newtheorem{theorem}{Theorem}[section]
\newtheorem{lemma}[theorem]{Lemma}
\newtheorem{proposition}[theorem]{Proposition}
\newtheorem{corollary}[theorem]{Corollary}

\newtheorem{remark}[theorem]{Remark}

\DeclarePairedDelimiter\abs{\lvert}{\rvert}%

\newcommand{\R}{\mathbb{R}}
\newcommand{\G}{\mathcal{G}}

\newcommand{\F}{\mathcal{F}}
\newcommand{\N}{\mathbb{N}}

\newcommand{\Eps}{\mathcal{E}}

\newcommand{\HH}{\mathcal{H}}
\newcommand{\K}{\mathcal{K}}

\newcommand{\f}{\frac}

\newcommand{\ee}{\mathcal{E}}

\newcommand\vv{\textsc{v}}
\newcommand\ww{\textsc{w}}

\tikzstyle{nodino}=[circle,draw,fill,inner sep=0pt,minimum size=0.5mm]
\tikzstyle{infinito}=[circle,inner sep=0pt,minimum size=0mm]
\tikzstyle{nodo}=[circle,draw,fill,inner sep=0pt, minimum size=0.5*width("k")]
\tikzstyle{nodo_vuoto}=[circle,draw,inner sep=0pt, minimum size=0.5*width("k")]
\usetikzlibrary{graphs}
\tikzset{every loop/.style={min distance=10mm,in=300,out=240,looseness=10}}
\tikzset{place/.style={circle,thick,draw=blue!75,fill=blue!20,minimum
		size=6mm}}
\tikzset{place2/.style={circle,thick,draw=red!75,fill=red!20,minimum
		size=6mm}}

\title{Doubly nonlinear Schr\"odinger ground states on metric graphs}
\author[F. Boni]{Filippo Boni}
\address[F. Boni]{Politecnico di Torino, Dipartimento di Scienze Matematiche "G.L. Lagrange", Corso Duca degli Abruzzi 24, 10129, Torino, Italy.}
\email{filippo.boni@polito.it}

\author[S. Dovetta]{Simone Dovetta}
\address[S. Dovetta]{Università degli Studi di Roma ``La Sapienza", Dipartimento di Scienze di Base ed Applicate per l'Ingegneria, via Antonio Scarpa 14, 00161 - Roma, Italy.}
\email{simone.dovetta@uniroma1.it}

\begin{document}
\maketitle
\begin{abstract}
	We investigate the existence of ground states at prescribed mass on general metric graphs with half--lines for focusing doubly nonlinear Schr\"odinger equations involving both a standard power nonlinearity and delta nonlinearities located at the vertices. The problem is proved to be sensitive both to the topology and to the metric of the graph and to exhibit a phenomenology richer than in the case of the sole standard nonlinearity considered in \cite{ASTcpde,AST}. On the one hand, we provide a complete topological characterization of the problem, identifying various topological features responsible for existence/non--existence of doubly nonlinear ground states in specific mass regimes. On the other hand, we describe the role of the metric in determining the exact interplay between these different topological properties. 
\end{abstract}


\section{Introduction and main results}
In this paper we investigate the existence of ground states for the doubly nonlinear Schr\"odinger energy functional
\begin{equation}
\label{Fpq1}
F_{p,q}(u,\G)=\f{1}{2}\int_{\G}|u'|^{2}\,dx-\f{1}{p}\int_{\G}|u|^{p}\,dx-\f{1}{q}\sum_{\vv\in V}|u(\vv)|^{q}
\end{equation}
under the mass constraint
\begin{equation*}
\int_{\G}|u|^{2}\,dx=\mu\,.
\end{equation*}
Here $\G=(V,E)$ is a non--compact metric graph with finitely many vertices $V$ and edges $E$, some of which unbounded. Given $\mu>0$, let
\begin{equation*}
H^{1}_{\mu}(\G):=\{u\in H^{1}(\G)\,:\,\|u\|_{L^{2}(\G)}^{2}=\mu\}
\end{equation*}
denote the mass--constrained space (for standard definition of functional spaces on graphs see for instance \cite{BK13}), and 
\begin{equation}
\label{eq:problem}
\F_{p,q}(\mu,\G):=\inf_{u\in H^{1}_{\mu}(\G)}F_{p,q}(u,\G)
\end{equation}
the ground state energy level at mass $\mu$. Accordingly, a ground state of $F_{p,q}$ at mass $\mu$ on $\G$ is defined as a global minimizer of \eqref{Fpq1} among all functions belonging to $H^{1}_{\mu}(\G)$, i.e. a function $u\in H^{1}_{\mu}(\G)$ such that $F_{p,q}(u,\G)=\F_{p,q}(\mu,\G)$.

In what follows, we will consider the regime where both the nonlinearities are $L^2$--subcritical, i.e.
\begin{equation}
\label{eq:pq_sub}
2<p<6,\quad 2<q<4\,.
\end{equation}
Our aim is to discuss the dependence of the ground states problem \eqref{eq:problem} both on the parameters $\mu, p, q$ and on topological and metric properties of the graphs. 

\vspace{.1cm}
Since the second half of the previous century, the analysis of differential models on metric graphs (or networks) has been witnessing a significant growth and it is nowadays a lively and rich research area. As a consequence, the literature in the field is already extremely wide and continues to increase, so that no attempt to provide a detailed overview of all the existing results will be done here. We limit ourselves to note that both linear and nonlinear problems have been addressed extensively. For the linear case, we refer to \cite{BKKM,BLS,EFK,FMN,HKMP,KKLM} and references therein for some of the most recent developments. In the nonlinear case, a prominent focus has been devoted to Schr\"odinger equations (see for instance \cite{AD,ADST,BMP,BDL20,BDL21,DGMP,DSTaim,DSTjlms,GKP,KP-JDE,NP,PS17,PSV} as well as the recent review \cite{ABR} and references therein), but other nonlinear models have been considered too (see \cite{MNS} for the KdV equation and \cite{BCT,BCT2} for the Dirac equation).

Particularly relevant for our discussion is the ground states problem 
\begin{equation}
	\label{eq:problemE}
	\ee(\mu,\G):=\inf_{u\in H_\mu^1(\G)}E(u,\G)
\end{equation}
for the nonlinear Schr\"odinger energy functional with the standard nonlinearity only
\begin{equation}
	\label{eq:E}
	E(u,\G):=\f12\int_\G|u'|^2\,dx-\f1p\int_\G|u|^p\,dx\,.
\end{equation}
On the real line $\G=\R$, it is well--known \cite{cazenave} that in the $L^2$--subcritical regime $p\in(2,6)$ problem \eqref{eq:problemE} has a unique (up to space translations) positive solution, the soliton $\phi_\mu\in H_\mu^1(\R)$, explicitly given by
\begin{equation}
	\label{eq:phimu}
	\phi_{\mu}(x)=\mu^{\alpha}\phi_{1}(\mu^{\beta}x),\qquad x\in\R\,,
\end{equation}
where $\phi_1\in H_1^1(\R)$ is the soliton at mass $\mu=1$ and 
\begin{equation}
	\label{alfabeta}
	 \alpha:=\f{2}{6-p},\qquad \beta:=\f{p-2}{6-p}\,.
\end{equation}
After first investigations on star graphs (see e.g. \cite{acfn_pl}), the behaviour of \eqref{eq:problemE} on general non--compact graphs with half--lines has been characterized in \cite{ASTcpde,AST} for $L^2$--subcritical nonlinearities $p\in(2,6)$ and in \cite{AST-CMP} for the $L^2$--critical regime $p=6$. In particular, it has been shown that whether ground states at a certain mass exist strongly depends both on topological and on metric properties of the graph. As it will be important in the following, we highlight that a general topological assumption, named Assumption (H), ruling out existence of ground states of $E$ at any mass $\mu$ is given in \cite[Section 2]{AST}. Such an assumption can be stated for instance as follows
\[
\text{(H)\qquad every point of the graph lies on a trail that contains two half--lines}
\]
(for other equivalent formulations of Assumption (H) see \cite{ASTparma}). Recall that a trail is a connected path in $\G$ in which every edge of the path is run through exactly once. Example of a graph fulfilling Assumption (H) is given in Figure \ref{fig:grafoH}.

\begin{figure}[t]
	\centering
	\includegraphics[width=.65\columnwidth]{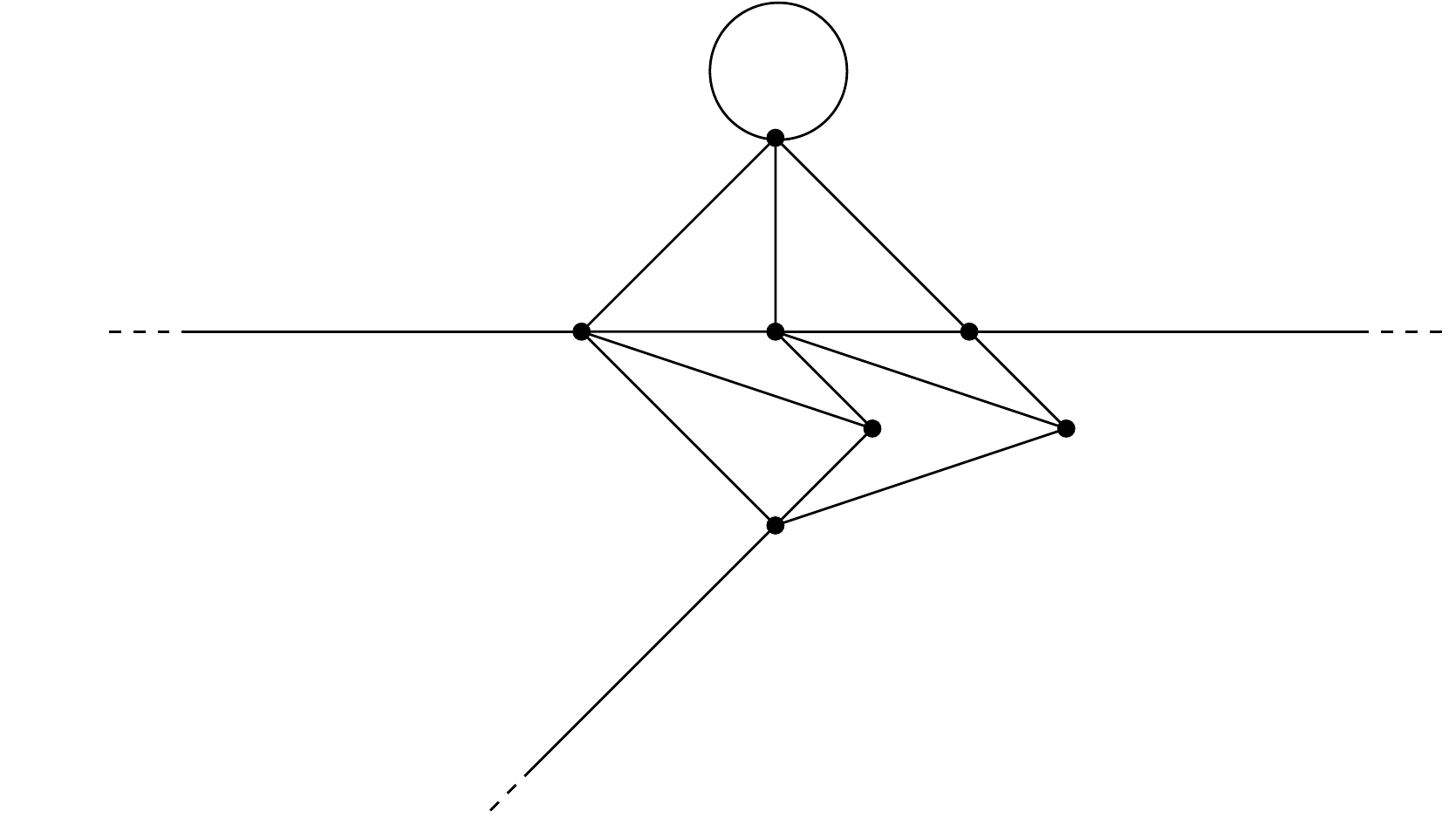}
	\caption{Example of a graph fulfilling Assumption (H) as in \cite{AST}.}
	\label{fig:grafoH}
\end{figure}

In the last years, the model with the sole standard nonlinearity has been generalized at least in two directions. 

On the one hand, \cite{PS20} recently addressed the ground states problem for the energy functional
\[
E_{6,p}(u,\G)=\f12\int_\G|u'|^2\,dx-\f16\int_\G|u|^6\,dx-\f\tau p\int_\G|u|^p\,dx\,,\qquad\tau\in\R,\,p\in(2,6),
\]
accounting for the combined effect of a $L^2$--critical and a (weighted) $L^2$--subcritical standard nonlinearity. It is shown how the interplay between the two standard nonlinearities sensibly affects the ground states problem, giving rise to new phenomena with respect to the single nonlinearity \eqref{eq:problemE}. This work seems to be the first paper on graphs fitting in the quite active research line of Schr\"odinger equations with combined standard nonlinearities (see for instance \cite{tao,jeanvisc,killip,lecoz,lewin,miao2,soave1,soave2}). 

On the other hand, fueled for instance by possible applications as the dynamic of confined charges \cite{jona,malomed} or the resonant tunneling \cite{nier}, concentrated nonlinearities have been proposed. In Euclidean spaces, models describing the effect of a delta potential have been considered first in dimension one and three \cite{ADFT1,ADFT2,AT,at_jfa} (see also the recent papers \cite{HL1,HL2}), whereas the analysis in dimension two is more recent \cite{ABCT,ACCT,ACCT2,CFN,CCT, FGI}. On non--compact graphs, the presence of a single standard nonlinearity restricted to the compact core of the graph (i.e. the union of all the bounded edges) has been discussed for instance in \cite{DT,ST-JDE, ST-NA,T-JMAA}. The first step towards the investigation of the interaction between a standard nonlinearity and a pointwise term can be found in \cite{acfn_jde,acfn_aihp}, where a linear delta potential is located at the unique vertex of a star graph. Such a model amounts to take $\G$ as a star graph and $q=2$ in \eqref{Fpq1}. In this context, ground states are proved to exist for small masses only, as they bifurcate from the corresponding solution of the associated linear problem. Since ground states for the model \eqref{eq:E} with the sole standard nonlinearity never exist on star graphs, this existence result is a first marker of the nontrivial interplay that occurs between a standard nonlinearity and a (linear) delta potential. 

In this paper, we push forward this analysis by considering a model somewhat on the edge between combined and concentrated nonlinearities: the energy functional \eqref{Fpq1} involving two focusing nonlinearities, a standard one and a pointwise one. The corresponding ground states problem \eqref{eq:problem} has already been addressed both on the real line \cite{BD} and on star graphs \cite{ABD}. On the real line, in the regime \eqref{eq:pq_sub} where both the nonlinearities are $L^2$--subcritical ground states exist for every value of the mass (see \cite[Theorem 1.3]{BD} and Section 2 below). The portrait is sensibly richer on star graphs. New threshold phenomena arise, concerning both the value of the mass and that of the exponents $p,q$. Precisely, if $q<\f p2+1$, then ground states exist if and only if the mass is smaller than a critical value, whereas if $q>\f p2+1$ the situation is reversed and ground states exist for large masses only (see \cite[Theorem 1.1]{ABD} as well as Section 2 below). Furthermore, if $q=\f p2+1$, then the existence of ground states is insensitive of the mass and depends only on how many half--lines appear in the graph. In particular, ground states exist for every mass on star graphs with a number of half--lines smaller than a threshold (depending on $p$), whereas they do not exist for any mass whenever the number of half--lines exceeds such a critical value \cite[Theorem 1.2]{ABD}.

\vspace{.1cm}
We can now state and discuss the main results of the present paper, that extend the analysis of the ground states problem \eqref{eq:problem} to general non--compact metric graphs with half--lines. 

Let us first highlight that, as a straightforward consequence of Corollary \ref{compactcor} below, existence of ground states of $F_{p,q}$ at mass $\mu$ is for free whenever it is already known that ground states at mass $\mu$ exist for the problem \eqref{eq:problemE} with the standard nonlinearity only (see Remark \ref{rem:exF_da_exE}). On the contrary, when ground states of $E$ do not exist, the ground states problem for $F_{p,q}$ is far from trivial. Since in general existence of such states depends deeply both on the topology and on the metric of $\G$ and may require a case by case analysis, it is natural to restrict our attention to graphs for which it is granted a priori that ground states of $E$ never exist, for any value of the mass. For this reason, though this will clearly not exhaust the class of graphs for which non--existence of ground states of $E$ occurs, in what follows, according to \cite[Example 2.4]{ASTcpde}, we will focus on graphs fulfilling Assumption (H) that are not isomorphic to the real line or to the so--called towers of bubbles (see \cite[Figure 3]{ASTcpde}). Note that, to ease the statement of our main results, we will always write that the graphs we are considering satisfy Assumption (H), being understood without further notice that we are excluding the line and the towers of bubbles.

Even though Assumption (H) is enough to prevent solutions of problem \eqref{eq:problemE}, this is no longer true in the doubly nonlinear case \eqref{eq:problem}. The first part of our analysis provides a complete topological characterization of the problem. We begin with the following result.
\begin{theorem}
\label{3half-deg3}
Let $\G$ be a non--compact graph satisfying Assumption (H) with at least $3$ half--lines and all vertices of degree greater than or equal to $3$. Then there exist two thresholds $\underline{\mu}_{p,q}:=\underline{\mu}_{p,q}(p,q,\G)$, $\overline{\mu}_{p,q}:=\overline{\mu}_{p,q}(p,q,\G)$, so that $0<\underline{\mu}_{p,q}\leq\overline{\mu}_{p,q}$ and 
\begin{itemize}
		\item[(i)] if $q<\frac{p}{2}+1$, then ground states of $F_{p,q}$ at mass $\mu$ exist if $\mu<\underline{\mu}_{p,q}$ and do not exist if $\mu> \overline{\mu}_{p,q}$ ;
		\item[(ii)] if $q>\frac{p}{2}+1$, then ground states of $F_{p,q}$ at mass $\mu$ exist if  $\mu> \overline{\mu}_{p,q}$ and do not exist if $\mu<\underline{\mu}_{p,q}$.
\end{itemize}	
\end{theorem}
The previous theorem says that graphs fulfilling Assumption (H), with at least 3 half--lines and no vertex of degree smaller than 3 behave essentially as star graphs, that are indeed the easiest example of graphs covered by Theorem \ref{3half-deg3}. We point out that the existence parts in Theorem \ref{3half-deg3} are not difficult to obtain and remain true even removing the hypotheses of at least 3 half--lines and all vertices with degree not smaller than 3 (see Propositions \ref{exsubdiag-smallmu}--\ref{deg3-largemu} below). Conversely, the proof of the non--existence statements is rather involved and requires new ideas. We will comment a bit more on this at the end of this section. Note that, by \cite[Theorem 1.1]{ABD}, if $\G$ is a star graph then $\underline{\mu}_{p,q}=\overline{\mu}_{p,q}$ for every $q\neq\f p2+1$. However, the analysis in \cite{ABD} heavily relies on the fact that on star graphs an explicit characterization of the critical points of $F_{p,q}$ in $H_\mu^1(\G)$ is available. Since this is clearly out of reach on general non--compact graphs, to understand whether, for every graph fulfilling the hypotheses of Theorem \ref{3half-deg3}, the two thresholds $\underline{\mu}_{p,q}$, $\overline{\mu}_{p,q}$ coincide seems to be a challenging open question. Notice also that Theorem \ref{3half-deg3} gives no information when $q=\f p2+1$.

Even though at first sight Theorem \ref{3half-deg3} may lead to think that no new phenomenon arises when considering networks more general than star graphs, this is actually not the case. The following theorems identify two topological features of non--compact graphs that are responsible for existence results with no analogue on star graphs.
\begin{theorem}
\label{3half-deg2}
Let $\G$ be a non--compact graph with at least a vertex of degree $2$. Then there exists $\widetilde{\mu}_{p,q}:=\widetilde{\mu}_{p,q}(p,q,\G)>0$ such that ground states of $F_{p,q}$ at mass $\mu$ exist for every $\mu\geq\widetilde{\mu}_{p,q}$.
\end{theorem}
\begin{theorem}
	\label{2half}
	Let $\G$ be a non--compact graph with exactly $2$ half--lines. Then there exists $\widehat{\mu}_{p,q}:=\widehat{\mu}_{p,q}(p,q,\G)>0$ such that ground states of $F_{p,q}$ at mass $\mu$ exist for every $\mu\leq\widehat{\mu}_{p,q}$.
\end{theorem}
Note first that neither Theorem \ref{3half-deg2} nor Theorem \ref{2half} require Assumption (H) to be fulfilled. This is because each of them refers to a topological property rooting for existence of ground states independently of any other features of the graph.

For graphs fulfilling Assumption (H), both Theorem \ref{3half-deg2} and Theorem \ref{2half} shed some light when each of the two hypotheses of Theorem \ref{3half-deg3} are removed. As already highlighted before, note that getting rid of such hypotheses do not affect the existence results stated in Theorem \ref{3half-deg3}.

On the one hand, the presence of a vertex of degree 2 accounts for existence of large masses ground states, regardless of the specific value of $p,q$. As the proof of Theorem \ref{3half-deg2} will display clearly, this can be seen as a reminiscence of the behaviour of the problem on the real line, where ground states always exist and concentrate around the origin as the mass increases. For graphs fulfilling Assumption (H) with at least 3 half--lines, this is particularly interesting in the regime $q<\f p2+1$. Indeed, combining Theorem \ref{3half-deg3}(i) with Theorem \ref{3half-deg2}, the following is immediate.
\begin{corollary}
	\label{cor:3half2deg}
	Let $\G$ be a non--compact graph satisfying Assumption (H), with at least $3$ half--lines and at least a vertex of degree $2$ (e.g. Figure \ref{fig:cor}(A)), and let $q<\frac{p}{2}+1$. Then there exist $\underline{\mu}_{p,q}:=\underline{\mu}_{p,q}(p,q,\G)>0$, $\widetilde{\mu}_{p,q}:=\widetilde{\mu}_{p,q}(p,q,\G)>0$, so that ground states of $F_{p,q}$ at mass $\mu$ exist both if $\mu<\underline{\mu}_{p,q}$ and if $\mu\ge \widetilde{\mu}_{p,q}$.
\end{corollary}
This marks a sharp difference with respect to star graphs, for which ground states at large masses never exist in the regime $q<\f p2+1$. Moreover, let us also stress the fact that the role of vertices of degree 2 is new and peculiar of the doubly nonlinear problem we are considering. Indeed, it is well--known that vertices of degree 2 are completely inessential when dealing with standard nonlinearities only.

On the other hand, Theorem \ref{2half} unravels a somewhat surprising new phenomenon, as the presence of exactly two half--lines is enough to guarantee that ground states at small masses always exist, independently of $p,q$. Even though it is not evident, also in this case the argument of the proof will show that such a phenomenon is rooted in the behaviour of ground states on the real line. The main idea underpinning this result is that, when the mass is sufficiently small, the doubly nonlinear problem does not feel any difference between a single delta concentrated at a point or finitely many of them located on a given compact core. Again, combining Theorem \ref{2half} with Theorem \ref{3half-deg3}(ii) and Theorem \ref{3half-deg2} has the next direct consequences, highlighting once more the rich structure of the problem.
\begin{corollary}
	\label{cor:2half3deg}
	Let $\G$ be a non--compact graph satisfying Assumption (H), with exactly $2$ half--lines and no vertex of degree $2$ (e.g. Figure \ref{fig:cor}(B)), and let $q>\f{p}{2}+1$. Then there exist $\widehat{\mu}_{p,q}:=\widehat{\mu}_{p,q}(p,q,\G)>0$, $\overline{\mu}_{p,q}:=\overline{\mu}_{p,q}(p,q,\G)>0$ such that ground states of $F_{p,q}$ at mass $\mu$ exist both if $\mu\le\widehat{\mu}_{p,q}$ and if $\mu> \overline{\mu}_{p,q}$.
\end{corollary}
\begin{corollary}
	\label{cor:2half2deg}
	Let $\G$ be a non--compact graph with exactly $2$ half--lines and at least a vertex of degree $2$ (e.g. Figure \ref{fig:cor}(C)). Then there exist $\widetilde{\mu}_{p,q}:=\widetilde{\mu}_{p,q}(p,q,\G)>0$, $\widehat{\mu}_{p,q}:=\widehat{\mu}_{p,q}(p,q,\G)>0$ such that ground states of $F_{p,q}$ at mass $\mu$ exist both if $\mu\le\widehat{\mu}_{p,q}$ and if $\mu\ge \widetilde{\mu}_{p,q}$.
\end{corollary}

\begin{figure}[t]
	\centering
	\subfloat[][A graph satisfying Assumption (H) with at least $3$ half--lines and at least a vertex with degree $2$.]{
	\includegraphics[width=0.4\textwidth]{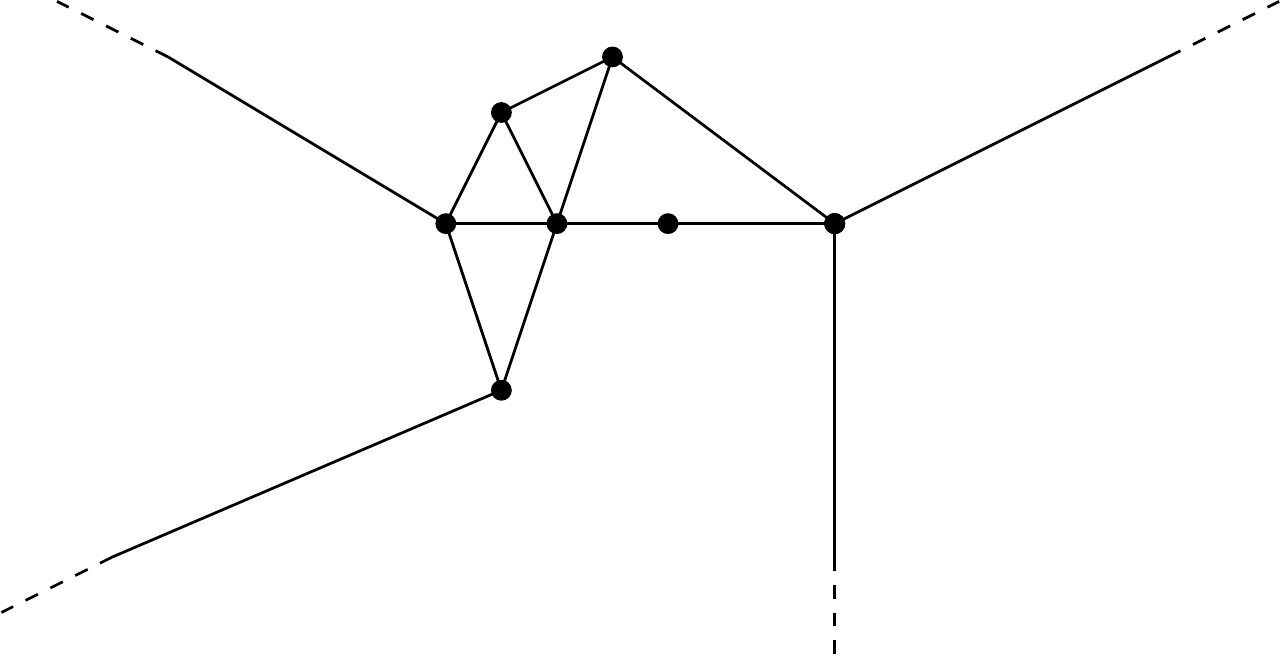}}
	\quad
	\subfloat[][A graph satisfying Assumption (H) with exactly $2$ half--lines and no vertex with degree smaller than $3$.]{
	\includegraphics[width=0.5\textwidth]{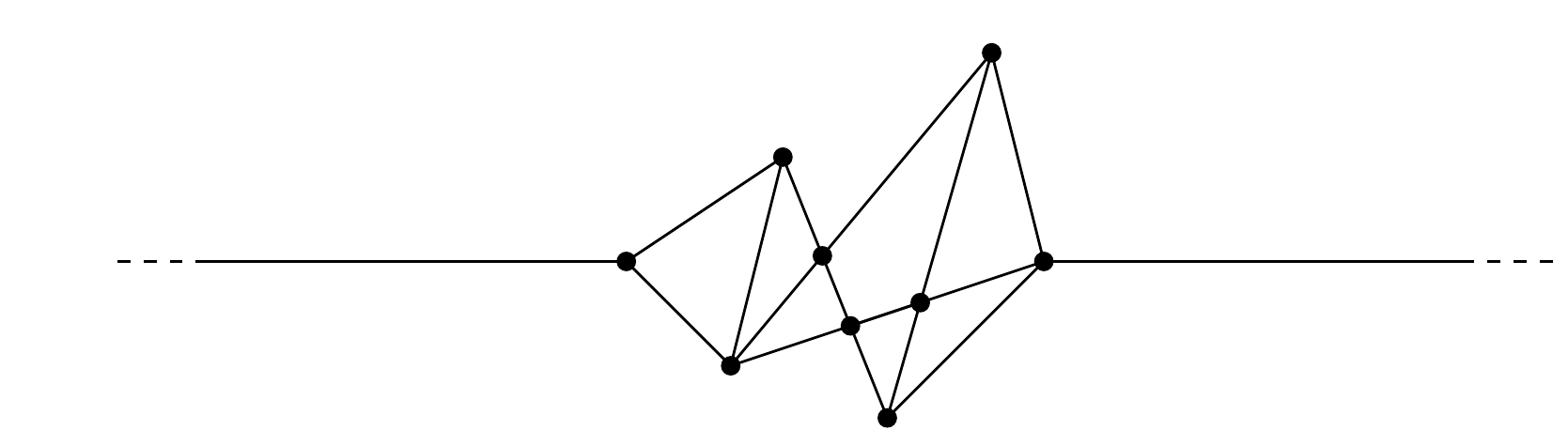}}
	
	\subfloat[][A graph with exactly $2$ half--lines and at least a vertex with degree $2$.]{
	\includegraphics[width=0.5\textwidth]{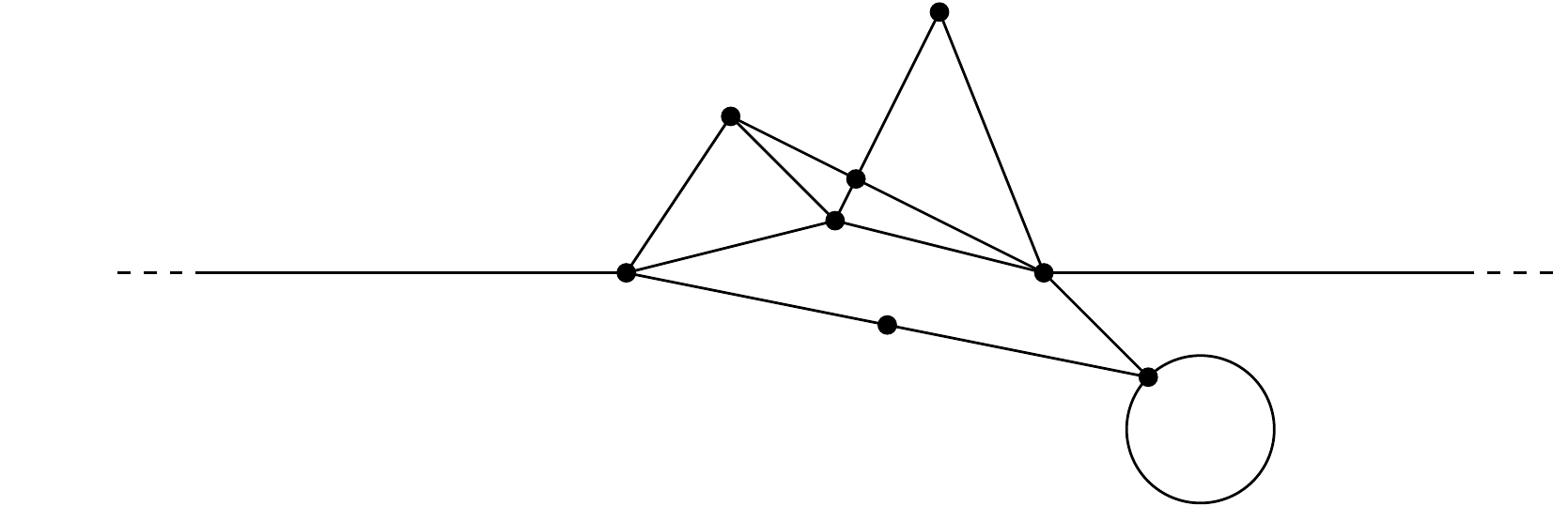}}

	\caption{Examples of graphs as in Corollaries \ref{cor:3half2deg}--\ref{cor:2half3deg}--\ref{cor:2half2deg}.}
	\label{fig:cor}
\end{figure}

The results discussed so far outline a comprehensive description of how the ground states problem on graphs with half--lines is affected by the topology of the network. However, the presence of various topological features as in Theorems \ref{3half-deg3}--\ref{3half-deg2}--\ref{2half}, each guaranteeing on its own the existence of ground states in different mass regimes, raises new questions. For instance, Corollary \ref{cor:3half2deg} provides a class of graphs where small masses ground states exist because of the similarity between general non--compact graphs and star graphs, whereas ground states at large masses exist due to vertices with degree 2, that make the problem on graphs resembles the corresponding one on the line. This seems to suggest that the thresholds $\underline{\mu}_{p,q}$, $\widetilde{\mu}_{p,q}$ are rooted in two unrelated properties of the graph. It is then natural to wonder if a general relation between these two values holds for every graph in Corollary \ref{cor:3half2deg} or if it depends on further properties of the graph. Analogous problems are posed by Corollary \ref{cor:2half3deg} and Corollary \ref{cor:2half2deg}. 

To partially answer this kind of questions, we have the following results.

\begin{theorem}
	\label{thm:metric1}
	Let $q<\f p2+1$. Then
	\begin{itemize}
		\item[(i)] there exists a graph $\G^1$ satisfying Assumption (H), with at least $3$ half--lines and at least a vertex of degree $2$, so that ground states of $F_{p,q}$ at mass $\mu$ exist for every $\mu>0$;
		\item[(ii)] there exists a graph $\G^2$ satisfying Assumption (H), with at least $3$ half--lines and at least a vertex of degree $2$, and a value $m>0$ so that ground states of $F_{p,q}$ at mass $m$ on $\G^2$ do not exist. 
	\end{itemize}
\end{theorem}

\begin{theorem}
	\label{thm:metric2}
	Let $q>\f p2+1$. Then there exists a graph $\G$ satisfying Assumption (H), with exactly $2$ half--lines and no vertex of degree $2$, and a value $m>0$ so that ground states of $F_{p,q}$ at mass $m$ on $\G$ do not exist. 
\end{theorem}

\begin{theorem}
	\label{thm:metric3}
	There exists a graph $\G$, with exactly $2$ half--lines and at least a vertex of degree $2$, and a value $m>0$ so that ground states of $F_{p,q}$ at mass $m$ on $\G$ do not exist.
\end{theorem}
Theorems \ref{thm:metric1}--\ref{thm:metric2}--\ref{thm:metric3} exploit the key role of the metric of the graph. On the one hand, Theorem \ref{thm:metric1} is proved by considering graphs with exactly one vertex with degree 2 (see Figure \ref{fig:metr1}) and investigating the dependence of ground states at prescribed mass on the length of the edges emanating from this vertex. It turns out that ground states always exist when these two edges are sufficiently long (Proposition \ref{prop:metrex}), whereas there are masses at which they do not exist when the edges emanating from the vertex of degree 2 are too short (Proposition \ref{prop:metrnonex1}). On the other hand, the proof of Theorem \ref{thm:metric2} is based on graphs with exactly 2 half--lines and whose compact core has total length way larger than its diameter (see Figure \ref{fig:metr2}). Theorem \ref{thm:metric3} then combines these features, since to exhibit a graph as in the statement of the theorem (Figure \ref{fig:metr3}) we add a vertex of degree 2 with two sufficiently short edges to the graph provided by Theorem \ref{thm:metric2}. 

Let us highlight that Theorem \ref{thm:metric1} is somewhat complete, showing that under the hypotheses of Corollary \ref{cor:3half2deg} there are both graphs for which ground states exist for every mass (as on the real line) and graphs where non--existence occurs at certain masses (akin to star--graphs). Conversely, Theorems \ref{thm:metric2}--\ref{thm:metric3} provide a partial information only, since at present we are not able to exhibit graphs fulfilling the hypotheses of Corollary \ref{cor:2half3deg} or Corollary \ref{cor:2half2deg} for which ground states always exist. Actually, it is not even clear to us whether graphs like this do exist. Roughly, the problem is the following. On the one hand, to keep the threshold $\widehat{\mu}_{p,q}$ in Corollaries \ref{cor:2half3deg}--\ref{cor:2half2deg} bounded away from zero, one would need the total length of the compact core to be not too large. On the other hand, to tune $\overline{\mu}_{p,q}$ of Corollary \ref{cor:2half3deg} and $\widetilde{\mu}_{p,q}$ of Corollary \ref{cor:2half2deg} to sufficiently small values, it seems to be necessary that the lengths of specific edges are large enough. It is an open problem to understand whether one can balance these conflicting features to obtain existence of ground states at every mass.

To conclude this section, we would like to briefly comment on the proof of Theorem \ref{3half-deg3}. In fact, the existence results follow easily by adapting techniques already used on star graphs. Conversely, the proof of the non--existence statement is more involved and technically demanding. In a nutshell, our argument combines the available knowledge on star graphs with a new way to exploit the natural scaling of the problem on general graphs. Here we just sketch the main idea, redirecting to Section \ref{sec:nonex} below for the details. 

On the one hand, if $u\in H_\mu^1(\G)$ is a ground state of $F_{p,q}$ at mass $\mu$, then by Corollary \ref{compactcor} below its doubly nonlinear energy satisfies
\[
F_{p,q}(u,\G)\leq\ee(\mu,\R),
\]
which gives the following upper bound on the standard energy $E$ of $u$
\begin{equation}
\label{upper}
E(u,\G)\leq\ee(\mu,\R)+\f1q\sum_{\vv\in V}|u(\vv)|^q\,.
\end{equation}
On the other hand, since we are considering graphs fulfilling Assumption (H), it is well--known \cite{AST} that for every $u\in H_\mu^1(\G)$
\[
E(u,\G)>\ee(\mu,\R)\,.
\]
The key point in the proof of the non--existence results of Theorem \ref{3half-deg3} is to improve this lower bound in a quantitative version like
\begin{equation}
\label{lower}
E(u,\G)\geq\ee(\mu,\R)+\mathcal{R}(u),
\end{equation}
for some suitable remainder term $\mathcal{R}$ depending on $u$. To obtain such a remainder, we rely on the invariance of the normalized energy $E(u,\G)/\mu^{2\beta+1}$ under the scaling
\[
\G\mapsto t^{-\beta}\G,\quad u(\cdot)\mapsto t^\alpha u(t^\beta\cdot),\quad t>0, 
\]
where $\alpha,\beta$ are as in \eqref{alfabeta} (see Remark \ref{rem:omot}). This allows to pass from the upper bound \eqref{upper} for functions at mass $\mu$ on $\G$ to the upper bound
\begin{equation}
\label{upper2}
E(w,\G_\mu)\leq\ee(m,\R)+\left(\f\mu m\right)^{\alpha q-(2\beta+1)}\f1q\sum_{\vv\in V}|w(\vv)|^q
\end{equation}
for functions $w$ (the scaled version of $u$ with $t=m/\mu$) at mass $m$ on $\G_\mu:=\f m\mu\G$, for any desired mass $m>0$. This upper bound involves a weight depending on the original mass $\mu$ and on the quantity $\alpha q-(2\beta+1)$, whose sign depends on $q$ being smaller or larger than $\f p2+1$. Our argument will then proceed as follow. We will first assume by contradiction that ground states exist in those regimes where Theorem \ref{3half-deg3} asserts non--existence. Hence, taking sequences of ground states indexed by the mass (i.e. ground states for $\mu\to0$ if $q>\f p2+1$ and ground states for $\mu\to+\infty$ if $q<\f p2+1$), we will exploit the scaling to construct sequences of functions at a prescribed mass $m$, choosing $m$ so that ground states of $F_{p,q}$ at mass $m$ do not exist on star graphs. This leaves us with a sequence of functions at mass $m$ supported on $\G_\mu$. In particular, the lengths of the bounded edges will vary according to $\mu$. This will be crucial to obtain a scaled version of the lower bound with remainder \eqref{lower} for $w$, i.e. an inequality in the form
\begin{equation*}
	E\left(w,\G_\mu\right)\geq\ee(m,\R)+\mathcal{R}(\mu,w)
\end{equation*}
that, combined with \eqref{upper2}, will provide the contradiction we seek. 

\vspace{.1cm}
The paper is organized as follows. Section \ref{sec:prel} collects some preliminary results on the doubly nonlinear ground states problem on graphs. Sections \ref{sec:ex}--\ref{sec:nonex} provide the proof of the topological results: in Section \ref{sec:ex} we consider the existence statements of Theorem \ref{3half-deg3} and Theorems \ref{3half-deg2}--\ref{2half}, whereas in Section \ref{sec:nonex} we conclude the proof of Theorem \ref{3half-deg3} with the non--existence results. Section \ref{sec:metr} discusses the role of the metric and gives the proof of Theorems \ref{thm:metric1}--\ref{thm:metric2}--\ref{thm:metric3}. Finally, Appendices \ref{sec:appA}--\ref{sec:appB} contain some technical results used here and there in the paper.

\smallskip
\textbf{Notation.} Throughout, norms will be denoted avoiding the domain of integration whenever possible (e.g. $\|u\|_p$ will stand for $\|u\|_{L^p(\G)}$), using full notation only where necessary.

\section{Preliminaries}
\label{sec:prel}
In this section we begin with some preliminary results that will be important in our discussion. 

To start with, let us introduce some notation we will stick to all along the paper.  
We recall that here a non--compact metric graph $\G=(V,E)$ is a metric graph for which both the set of vertices $V$ and the set of edges $E$ are finite, but $E$ contains both bounded and unbounded edges. As usual, each bounded edge is identified with a bounded interval, while every unbounded edge $\HH_i$, $i=1,\, \dots,N$, is identified with (a copy of) the half--line.
The set of all the bounded edges of $\G$ will be called its compact core and will be denoted by $\K$. Whenever needed, we will use $\ell:=|\K|$ for the total length of the compact core. Moreover, we will denote by $n:=\#V$ the total number of vertices of $\G$ and by $V^{+}\subseteq V$ the subset of vertices $\vv\in V$ attached to at least one unbounded edge.

\subsection{Ground states on general non--compact graphs: properties and existence criteria}
Let us first recall the following Gagliardo--Nirenberg inequalities 
\begin{align}
\|u\|_{p}^{p}&\le K_{p}\|u\|_{2}^{\f{p}{2}+1}\|u'\|_{2}^{\f{p}{2}-1},\qquad p>2\,,\label{GnpG}\\
\|u\|_{\infty}^{2}&\le K\|u\|_{2}\|u'\|_{2},\label{GninfG}
\end{align}
holding on every non--compact graph $\G$, for every $u\in H^1(\G)$. Here $K_p,K$  are positive constants depending on $p$ and $\G$ only.

Observe also that, for every $u\in H_\mu^1(\G)$, it is evident that $F_{p,q}(u,\G)=E(u,\G)-\f1q\sum_{\vv\in V}|u(\vv)|^q\leq E(u,\G)$. Hence, for every graph $\G$ and mass $\mu$ we have
\begin{equation}
\label{eq:F leq E}
\F_{p,q}(\mu,\G)\leq\ee(\mu,\G)\,.
\end{equation}
If $\G$ is a non--compact graph with half--lines, \cite[Theorem 2.2]{ASTcpde} shows that $\ee(\mu,\G)\leq\ee(\mu,\R)=E(\phi_\mu,\R)$. Moreover, according to \eqref{eq:phimu} one has
\begin{equation}
	\label{Ephimu}
	E(\phi_\mu,\R)=-\theta_p\mu^{2\beta+1},\qquad\theta_p:=-E(\phi_1,\R)>0
\end{equation}
where $\beta$ is as in \eqref{alfabeta}. Summing up, for every non--compact graph $\G$ and mass $\mu$ it always holds
\begin{equation}
	\label{FGmu}
	\F_{p,q}(\mu,\G)\leq-\theta_p\mu^{2\beta+1}\,.
\end{equation}
The next lemma provides a priori estimates for functions with energy $F_{p,q}$ sufficiently close to the ground state level $\F_{p,q}$. Similar estimates were obtained in \cite[Lemma 2.6]{ASTcpde} for the problem \eqref{eq:problemE} with the standard nonlinearity only, for every value of the mass $\mu$. In our setting, it is possible to recover these results in certain mass regimes only.
\begin{remark}
	\label{rem:diagpq}
	Since it will be frequently used in the following, we recall here that, given $\alpha,\beta$ as in \eqref{alfabeta}, the relation $\alpha q< 2\beta+1$ holds if and only if $q<\f p2+1$.
\end{remark}
\begin{lemma}
\label{limitest-norms}
Let $\G$ be a non--compact graph and $p\in(2,6)$, $q\in(2,4)$. Then there exists $\mu^*:=\mu^*(p,q,\G)>0$ so that
\begin{itemize}
	\item[(i)] for every $\mu\geq\mu^*$ if $q<\f{p}{2}+1$, or
	\item[(ii)] for every $\mu\leq\mu^*$ if $q>\f{p}{2}+1$,
\end{itemize}
if $u\in H^{1}_{\mu}(\G)$ satisfies $F_{p,q}(u,\G)\le\f{1}{2}\F_{p,q}(\mu,\G)$, then
\begin{equation}
\label{kinen-mu}
C_{p}^{-1}\mu^{2\beta+1}\le \|u'\|_2^{2} \le C_{p}\mu^{2\beta+1},
\end{equation}
\begin{equation}
\label{pnorm-mu}
C_{p}^{-1}\mu^{2\beta+1} \le \|u\|_p^{p} \le C_{p}\mu^{2\beta+1},
\end{equation}
\begin{equation}
\label{infnorm-mu}
C_{p}^{-1}\mu^{\alpha}\le \|u\|_{\infty} \le C_{p}\mu^{\alpha}
\end{equation}
for some constant $C_{p}>0$ depending on $p$ and $\G$ only.
\end{lemma}
\begin{proof}
We split the proof in two parts.

\smallskip
\noindent\textit{Part 1: $q<\f p2+1$.} We need to prove that there exists $\mu^*>0$ so that \eqref{kinen-mu}--\eqref{pnorm-mu}--\eqref{infnorm-mu} hold for every $\mu\geq\mu^*$. Let $u\in H^{1}_{\mu}(\G)$ be such that $F_{p,q}(u,\G)\le\f{1}{2}\F_{p,q}(\mu,\G)$ and consider the notation $T:=\|u'\|_{2}^{2}$, $P:=\|u\|_{p}^{p}$ and $D:=\|u\|_{\infty}^2$, so that combining with \eqref{FGmu} gives
\begin{equation}
\label{FGmu1}
\f{1}{2}T-\f{1}{p}P-\f{1}{q}\sum_{\vv\in V}|u(\vv)|^{q}\le-\f{\theta_{p}}{2}\mu^{2\beta+1},
\end{equation}
whereas \eqref{GnpG} and \eqref{GninfG} become respectively
\begin{equation}
\label{Gnp2}
P\le K_{p}\mu^{\f{p+2}{4}}T^{\f{p-2}{4}}
\end{equation}
and
\begin{equation}
\label{Gninf2}
D\le K \mu^{\f{1}{2}}T^{\f{1}{2}}.
\end{equation}
We start by proving the upper bound in \eqref{kinen-mu}. By contradiction, assume that there exists a subsequence $\mu_{k}\to+\infty$ as $k\to +\infty$ such that 
\begin{equation}
\label{Tmuinf}
\lim_{k\to+\infty}\f{T}{\mu_{k}^{2\beta+1}}=+\infty.
\end{equation}
By \eqref{FGmu1}, \eqref{Gnp2} and \eqref{Gninf2} (and recalling \eqref{alfabeta})
\begin{equation*}
\f{1}{2}T\le\f{K_p}{p}\mu_k^{\f{p+2}{4}}T^{\f{p-2}{4}}+\f{nK^{\f q2}}{q}\mu_k^{\f{q}{4}}T^{\f{q}{4}},
\end{equation*}
hence
\begin{align*}
\f{T}{\mu_k^{2\beta+1}}&\le \f{2K_p}{p}\mu_k^{\f{p+2}{4}-(2\beta+1)}T^{\f{p-2}{4}}+\f{2nK^{\f q2}}{q}\mu_k^{\f{q}{4}-(2\beta+1)}T^{\f{q}{4}}\\
&=\f{2K_p}{p}\left(\f{T}{\mu_k^{2\beta+1}}\right)^{\f{p-2}{4}}+\f{2nK^{\f q2}}{q}\left(\f{T}{\mu_k^{2\beta+1}}\right)^{\f{q}{4}}\mu_k^{\alpha q-(2\beta+1)}\,.
\end{align*}
Dividing the last inequality by $\f{T}{\mu_k^{2\beta+1}}$ then yields
\begin{equation}
\label{contrTmu}
1\le \f{2K_p}{p}\left(\f{T}{\mu_k^{2\beta+1}}\right)^{-\f{6-p}{4}}+\f{2nK^{\f q2}}{q}\left(\f{T}{\mu_k^{2\beta+1}}\right)^{-\f{4-q}{4}}\mu_k^{\alpha q-(2\beta+1)}.
\end{equation}
Since $p<6$, $q<4$ and $q<\f{p}{2}+1$, by Remark \ref{rem:diagpq} the right hand--side in \eqref{contrTmu} goes to zero as $k\to+\infty$, providing the contradiction we seek. Thus, there exists $\mu^*>0$ so that the upper bound in \eqref{kinen-mu} holds for every $\mu\geq\mu^*$. Combining with \eqref{Gnp2}--\eqref{Gninf2}, we also get for every $\mu\geq\mu^*$
\begin{align*}
P\le& K_{p}'\mu^{2\beta+1}\\
D\le& K' \mu^{\beta+1}\,,
\end{align*}
which prove the upper bounds in \eqref{pnorm-mu}--\eqref{infnorm-mu} (note that $2\alpha=\beta+1$).

As for the lower bounds, assume by contradiction that there exists a subsequence $\mu_{k}\to +\infty$ as $k\to +\infty$ such that 
\begin{equation}
\label{Vmu0}
\lim_{k\to+\infty}\f{P}{\mu_{k}^{2\beta+1}}=0.
\end{equation}
By \eqref{FGmu1} and the fact that $T\ge 0$, we already know that
\begin{equation*}
\f{1}{p}\f{P}{\mu_k^{2\beta+1}}+\f{1}{q}\f{\sum_{\vv\in V}|u(\vv)|^{q}}{\mu_k^{2\beta+1}}\ge \f{\theta_p}{2}
\end{equation*} 
and, taking advantage of the upper bound in \eqref{infnorm-mu}, this gives
\begin{equation}
\label{contrVmu}
\f{1}{p}\f{P}{\mu_k^{2\beta+1}}+\f{nC_{p}^q}{q}\mu_k^{\alpha q-(2\beta+1)}\ge \f{\theta_p}{2}.
\end{equation}
Therefore, as $\f{P}{\mu_k^{2\beta+1}}\to 0$, the left hand--side of \eqref{contrVmu} tends to zero by Remark \ref{rem:diagpq}, i.e. a contradiction. Consequently, up to possibly changing the value of $\mu^*$, the lower bound in \eqref{pnorm-mu} is proved for every $\mu\geq\mu^*$.
The lower bounds in \eqref{kinen-mu}--\eqref{infnorm-mu} follow by \eqref{Gnp2} and by the fact that $P\le ||u||_{\infty}^{p-2}\mu$ respectively.

\smallskip
\noindent\textit{Part 2: $q>\f p2+1$.} The argument is analogous to that in Part 1. We assume first by contradiction that there exists a subsequence $\mu_{k}\to 0$ as $k\to +\infty$ such that \eqref{Tmuinf} holds and show that this is impossible by Remark \ref{rem:diagpq} since $q>\f p2+1$. This is enough to prove the upper bounds in \eqref{kinen-mu}--\eqref{pnorm-mu}--\eqref{infnorm-mu} for every $\mu\leq\mu^*$, for some $\mu^*>0$. To prove the lower bounds we argue by contradiction, assuming the existence of a subsequence $\mu_{k}\to 0$ as $k\to +\infty$ such that \eqref{Vmu0} holds and adapting the previous argument to the case $q>\f p2+1$.
\end{proof}
\begin{remark}
Note that the condition $F_{p,q}(u,\G)\leq\f12\F_{p,q}(\mu,\G)$ is non--empty, as $\F_{p,q}(\mu,\G)<0$ by \eqref{FGmu} for every $\mu>0$.
\end{remark}
With the following lemma, we provide another qualitative property of doubly nonlinear ground states.
\begin{lemma}
	\label{dec-2half}
	Let $\G$ be a non--compact graph and $u\in H^{1}_{\mu}(\G)$ be a ground state of $F_{p,q}$ at mass $\mu$ on $\G$. If $\vv\in V^{+}$ is a vertex attached to $N\ge2$ half--lines, then $u$ is symmetrically decreasing on the $N$ half--lines emanating from $\vv$.
\end{lemma}
\begin{proof}
	Clearly, when there is no vertex of $\G$ attached to more than one half--line, there is nothing to prove. Let then $\vv\in V^{+}$ be a vertex with $N$ half--lines $(\HH_{i})_{i=1}^N$ emanating from it. Note that $\bigcup_{i=1}^N\HH_{i}$ can be interpreted as a (copy of a) star graph $S_{N}$ with $N$ half--lines. Hence, arguing as in the proof of \cite[Lemma 3.4]{ABD}, if $u\in H_\mu^1(\G)$ is a ground state of $F_{p,q}$ at mass $\mu$ on $\G$, then its restriction $u_{|\bigcup_{i=1}^N\HH_{i}}$ to $\bigcup_{i=1}^N\HH_{i}$ is either symmetric with respect to $\vv$ and monotonically decreasing on each half--line, or it is symmetric with respect to $\vv$ and monotonically decreasing on $N-1$ half--lines and on the remaining one, say $\HH_{1}$, it is non--decreasing from the origin to a unique maximum point and then non--increasing on the rest of $\HH_1$. To prove the lemma, we are thus left to rule out the latter case. To this end, assume by contradiction that $u$ is as in the second case. Let $u_1$ be the restriction of $u$ to $\HH_{1}\cup\HH_{2}$ and $u_2$ be the restriction of $u$ to $\G\setminus\left(\HH_1\cup\HH_2\right)$, and set 
	\begin{equation*}
	\mu_1:=\int_{\HH_1\cup\HH_2}\abs{u_1}^2\,dx,\quad \mu_2:=\int_{\G \setminus(\HH_{1}\cup\HH_2)}\abs{u_2}^2\,dx.
	\end{equation*}
	For $\varepsilon>0$ small enough, let $\nu\in(-\varepsilon,\varepsilon)$. Since by assumption $u$ is increasing at $\vv$ on $\HH_1$, it has no maximum point at $\vv$, so it is possible to define $u_\nu:\G\to\R$
	\begin{equation*}
	u_\nu(x):=\begin{cases}
	\sqrt{\frac{\mu_1+\nu}{\mu_1}}u_1(x+T(\nu)) & x\in\HH_{1}\cup\HH_2\\
	\sqrt{\frac{\mu_2-\nu}{\mu_2}}u_2(x) & x\in \G\setminus(\HH_1\cup\HH_2)\,,
	\end{cases} 
	\end{equation*}
	in such a way that the shift $T(\nu)$ satisfies $T(0)=0$ and $u_\nu$ is continuous at $\vv$. It then follows that $u_\nu\in H_\mu^1(\G)$ for every $\nu$ and
	\begin{equation*}
	\begin{split}
	\frac{d^2}{d\nu^2}F_{p,q}(u_\nu,\G)\Big|_{\nu=0}=&-\frac{p-2}{4}\left[\f{1}{\mu_1^{2}}\int_{\HH_{1}\cup \HH_{2}}\abs{u_1}^p\,dx+\f{1}{\mu_2^{2}}\int_{\G\setminus(\HH_{1}\cup \HH_{2})}\abs{u_2}^p\,dx\right]\\
	&-\frac{q-2}{4\mu_2^{2}} \sum_{\ww\in V}\abs{u(\ww)}^q\\
	\le& -\frac{p-2}{4\mu^{2}}\int_{\G}\abs{u}^p\,dx-\frac{q-2}{4\mu_2^{2}}\sum_{\ww\in V} \abs{u(\ww)}^q<0,
	\end{split}
	\end{equation*}
	so that, choosing $\varepsilon$ small enough, we get a contradiction with the fact that $u$ is a ground state of $F_{p,q}$ at mass $\mu$ on $\G$.
\end{proof}
The next proposition establishes an existence criterion for ground states of $F_{p,q}$ at prescribed mass.
\begin{proposition}
\label{compactth}
Let $\G$ be a non--compact graph and $\mu>0$. If $\F_{p,q}(\mu,\G)<\Eps(\mu,\R)$, then ground states of $F_{p,q}$ at mass $\mu$ exist.
\end{proposition}
\begin{proof}
	The proof is almost identical to that of \cite[Proposition 3.1]{ABD}, so that here we just sketch the argument to stress the unique minor modification that is needed. Let $(u_n)\subset H_\mu^1(\G)$ be a minimizing sequence for $F_{p,q}$, i.e. $F_{p,q}(u_n,\G)\to\F_{p,q}(\mu,\G)$ as $n\to+\infty$. By \eqref{GnpG}--\eqref{GninfG}, $u_n\rightharpoonup u$ in $H^1(\G)$ and $u_n\to0$ in $L_{\text{loc}}^\infty(\G)$, for some $u\in H^1(\G)$. Arguing as in \cite[Proposition 3.1]{ABD}, one obtains that either $u\equiv0$ on $\G$ or $u$ is a ground state at mass $\mu$. To rule out the former case, it is enough to note that if $u_n\to0$ in $L_{\text{loc}}^\infty(\G)$, then $u_n\to0$ in $L^\infty(\K)$, so that both $\sum_{\vv\in V}|u_n(\vv)|^q\to0$ and $\|u_n\|_{L^p(\K)}\to0$ as $n\to+\infty$. Hence, arguing as in the proof of \cite[Theorem 3.3]{AST} leads to
	\[
	\ee(\mu,\R)>\F_{p,q}(\mu,\G)=\lim_{n} F_{p,q}(u_n,\G)\geq\liminf_{n\to+\infty}E(u_n,\G)\geq\ee(\mu,\R),
	\]
	i.e. a contradiction.
\end{proof}
\begin{corollary}
\label{compactcor}
Let $\G$ be a non--compact graph and $\mu>0$. If there exists $u\in H^1_\mu(\G)$ such that $F_{p,q}(u,\G)\le \Eps(\mu,\R)$, then ground states of $F_{p,q}$ at mass $\mu$ exist.
\end{corollary}
\begin{proof}
	It is a straightforward consequence of Proposition \ref{compactth}.
\end{proof}
\begin{remark}
	\label{rem:exF_da_exE}
	By \eqref{eq:F leq E} and Corollary \ref{compactcor}, it follows immediately that if ground states of $E$ at mass $\mu$ exist on $\G$, then also ground states of $F_{p,q}$ at mass $\mu$ exist. Indeed, by \cite[Theorem 2.2]{AST}, if $u\in H_\mu^1(\G)$ is a ground state of $E$ in $H_\mu^1(\G)$, then necessarily $E(u,\G)\leq\ee(\mu,\R)$, so that $F_{p,q}(u,\G)\leq E(u,\G)\leq\ee(\mu,\R)$ and ground states of $F_{p,q}$ exist too.
\end{remark}

\subsection{Ground states on the real line and on star graphs}
To conclude this preliminary section, we report here some results for the doubly nonlinear problem on the real line and on star graphs. Almost all of the following has already been proved or it is a minor modification of the analysis in \cite{ABD,BD}. For the sake of completeness, the proof of what is new is provided here whenever needed.

The first result concerns the problem on the real line. 
\begin{proposition}
\label{exline}
Let $p\in(2,6)$, $q\in(2,4)$. On the real line $\R$, for every $\mu>0$ there exists a unique positive ground state $\eta_{\mu}\in H_\mu^1(\R)$ of $F_{p,q}$ at mass $\mu$. Moreover, $\eta_\mu$ satisfies 
\begin{align}
F_{p,q}(\eta_{\mu},\R)<&\,E(\phi_{\mu},\R),\label{Fpq<Ephi}\\
|\eta_{\mu}(0)|>&\,|\phi_{\mu}(0)|\,, \label{etamu>phimu}\\
|\eta_{\mu}(0)|^{2}=&\,o(\mu)\quad\text{as}\quad \mu\to 0.\label{infnorm-opic}
\end{align} 
\end{proposition}
\begin{proof} 
The existence of a unique positive ground state $\eta_\mu$ on $\R$ is the content of \cite[Theorem 1.3]{BD}. Moreover, computing the Euler--Lagrange equation of the problem shows that $\eta_\mu$ cannot coincide with the soliton $\phi_{\mu}$, so that
\begin{equation*}
F_{p,q}(\eta_{\mu},\R)=\F_{p,q}(\mu,\R)<F_{p,q}(\phi_{\mu},\R)=E(\phi_{\mu},\R)-\f{1}{q}|\phi_{\mu}(0)|^{q}<E(\phi_{\mu},\R)\,,
\end{equation*}
entailing \eqref{Fpq<Ephi}. On the contrary, since the soliton $\phi_\mu$ is the unique (up to translations) positive ground state of $E$ at mass $\mu$ on $\R$,
\[
E(\eta_{\mu},\R)\ge E(\phi_{\mu},\R)\,.
\] 
Combining with \eqref{Fpq<Ephi} gives \eqref{etamu>phimu}.

We are left to prove \eqref{infnorm-opic}. Observe first that $F_{p,q}(\eta_\mu,\R)<0$ for every $\mu>0$, by \eqref{Fpq<Ephi} and $E(\phi_\mu,\R)<0$. Coupling with \eqref{GnpG}--\eqref{GninfG}, it follows that
\begin{equation*}
\f{1}{2}\|\eta_{\mu}'\|_{2}^{2}-\f{K_{p}}{p}\mu^{\f{p+2}{4}}\|\eta_{\mu}'\|_{2}^{\f{p}{2}-1}-\f{K^{\f q2}}{q}\mu^{\f{1}{2}}\|\eta_{\mu}'\|_{2}^{\f q2}<0,
\end{equation*}
that entails the existence of a constant $M>0$ so that $\|\eta_\mu'\|_2\leq M$ for every $\mu$ small enough. In particular, this implies $\|\eta_\mu\|_\infty\to0$ by \eqref{GninfG} and thus $\|\eta_\mu\|_p^p\leq\|\eta_\mu\|_\infty^{p-2}\mu=o(\mu)$  as $\mu\to0$. Suppose then by contradiction that there exists $C>0$ such that $\|\eta_{\mu}'\|_{2}^{2}\ge C\mu$ as $\mu \to 0$. Since $F_{p,q}(\eta_{\mu})<0$, by \eqref{GninfG} we have
\begin{equation*}
\f{1}{2}\|\eta_{\mu}'\|_{2}^{2}+o\left(\|\eta_{\mu}'\|_{2}^{2}\right)=\f{1}{2}\|\eta_{\mu}'\|_{2}^{2}-\f1p\|\eta_\mu\|_p^p<\f1q|\eta_\mu(0)|^q\le\f{K^\f q2}{q}\mu^{\f{q}{4}}\|\eta_{\mu}'\|_{2}^{\f{q}{2}},
\end{equation*}
that, dividing by $\|\eta_{\mu}'\|_{2}^{\f{q}{2}}$ and using $\|\eta_{\mu}'\|_{2}^{2}\ge C\mu$, implies 
\begin{equation*}
1\leq C'\mu^{\f{2q-4}4}\quad\text{as}\quad\mu\to 0\,.
\end{equation*}
Since $q>2$, this is impossible. Hence, $\|\eta_{\mu}'\|_{2}=o(\sqrt{\mu})$ and (by \eqref{GninfG} again)
\begin{equation*}
|\eta_{\mu}(0)|^{2}=\|\eta_{\mu}\|_{\infty}^{2}\le K\sqrt{\mu}\|\eta_{\mu}'\|_{2}=o(\mu)\quad\text{as}\quad \mu\to 0\,.\qedhere
\end{equation*}
\end{proof}
The last result of the section is a generalization of \cite[Theorem 1.1]{ABD} on star graphs $S_N$ to the functional $F_{p,q,\tau}(\cdot,S_{N}):H^{1}_{\mu}(S_{N})\to\R$
\begin{equation}
\label{Fpq2}
F_{p,q,\tau}(u,S_{N})=\f{1}{2}\int_{S_{N}}|u'|^{2}\,dx-\f{1}{p}\int_{S_{N}}|u|^{p}\,dx-\f{\tau}{q}|u(0)|^{q},
\end{equation}
where $\tau>0$ is a positive parameter.
\begin{remark}
	The functional $F_{p,q,1}$ coincides with the usual functional $F_{p,q}$.
\end{remark}
\begin{proposition}
	\label{exstarg}
	Let $p\in(2,6),\,q\in(2,4)$, $q\neq\frac{p}{2}+1$ and $\tau>0$. Let $S_N$ be the star graph with $N\geq3$ half--lines. Then there exists a critical mass $\mu^{*}:=\mu^{*}(p,q,\tau,N)>0$ such that
	\begin{itemize}
		\item[(i)] if $q<\frac{p}{2}+1$, then ground states of \eqref{Fpq2} at mass $\mu$ exist if and only if $\mu\le\mu^{*}$;
		\item[(ii)] if $q>\frac{p}{2}+1$, then ground states of \eqref{Fpq2} at mass $\mu$ exist if and only if $\mu\ge\mu^{*}$.
	\end{itemize}	
	Furthermore, whenever they exist, ground states at prescribed mass are unique and radially decreasing on $S_N$, i.e. their restriction to each half--line of the graph corresponds to the same decreasing function on $\R^+$.
\end{proposition}
\begin{proof}
 The statement for $\tau=1$ has been proved in \cite[Theorem 1.1]{ABD}. It is immediate to verify that the same argument applies without changes to any choice of the parameter $\tau>0$.
\end{proof}

\section{Existence results: proof of Theorems \ref{3half-deg3}--\ref{3half-deg2}--\ref{2half}}
\label{sec:ex}

In this section we address the existence of doubly nonlinear ground states, providing the proof of the existence statements in Theorem \ref{3half-deg3}, as well as the proof of Theorems \ref{3half-deg2}--\ref{2half}.

We begin with the existence results in Theorem \ref{3half-deg3}, whose proof is a straightforward consequence of the next two propositions. Note that none of them requires the validity of Assumption (H), so that they hold in a more general setting than the one of Theorem \ref{3half-deg3}.
\begin{proposition}
\label{exsubdiag-smallmu}
Let $\G$ be a non--compact graph and $q<\f{p}{2}+1$. Then there exists $\underline{\mu}_{p,q}:=\underline{\mu}_{p,q}(p,q,\G)>0$ such that ground states of $F_{p,q}$ at mass $\mu$ on $\G$ exist for every $\mu<\underline{\mu}_{p,q}$.
\end{proposition}
\begin{proof}
Consider $u\in H^{1}_{\mu}(\G)$ given by
\[
u(x):=\begin{cases}
\phi_{m}(x) & \text{if }x\in\HH_i,\text{ for some }i=1,\,\dots,N\\
\|\phi_m\|_{L^\infty(\R)} & \text{ if }x\in\K\,,
\end{cases}
\]
where $\phi_m$ is the soliton at mass $m$ on $\R$ as in \eqref{eq:phimu}. Then it must be
\begin{equation*}
\mu=\f{N}{2}\|\phi_{m}\|_{L^{2}(\R)}^{2}+\ell\|\phi_{m}\|_{L^{\infty}(\R)}^{2}=\f{N}{2}m+\ell |\phi_1(0)|^2 m^{2\alpha}
\end{equation*}
where as usual $\ell:=|\K|$. In particular, observe that $m\to0$ as $\mu\to0$, so that $m^{2\alpha}=o(m)$ and $m=\mu+o(\mu)$ as $\mu \to 0$. Moreover, since $q<\f p2+1$, by \eqref{alfabeta} and Remark \ref{rem:diagpq} it holds $\alpha q<2\beta+1<\alpha p$. Hence, 
\begin{equation*}
\begin{split}
F_{p,q}(u,\G)&=\f{N}{2}E(\phi_{m},\R)-\f{\ell |\phi_1(0)|^p}{p}m^{\alpha p}-\f{n|\phi_1(0)|^q}{q}m^{\alpha q}\\
&=-\f{N}{2}\theta_p m^{2\beta+1}-\f{\ell |\phi_1(0)|^p}{p}m^{\alpha p}-\f{n |\phi_1(0)|^q}{q}m^{\alpha q}\\
&=-\f{n|\phi_1(0)|^q}{q}\mu^{\alpha q}+o(\mu^{\alpha q})\quad\text{as}\quad \mu\to 0\,,
\end{split}
\end{equation*}
so that
\begin{equation*}
F_{p,q}(u,\G)=-\f{n|\phi_1(0)|^q}{q}\mu^{\alpha q}+o(\mu^{\alpha q})\leq-\theta_{p}\mu^{2\beta+1}=\ee(\mu,\R)\quad\text{as }\mu\to0\,.
\end{equation*}
Therefore, by Corollary \ref{compactcor} there exists $\underline{\mu}_{p,q}>0$ such that ground states of $F_{p,q}$ at mass $\mu$ on $\G$ exist for every $\mu<\underline{\mu}_{p,q}$.
\end{proof}
\begin{proposition}
	\label{deg3-largemu}
	Let $\G$ be a non--compact graph with at least one vertex of degree greater than or equal to $3$ and $q>\f{p}{2}+1$. Then there exists $\overline{\mu}_{p,q}:=\overline{\mu}_{p,q}(p,q,\G)>0$ such that ground states of $F_{p,q}$ at mass $\mu$ on $\G$ exist for every $\mu>\overline{\mu}_{p,q}$.
\end{proposition}
\begin{proof}
	Let $\vv\in V$ be a vertex of degree not smaller than 3. Denote by $\{e_{i}\}_{i=1,\dots,N}$ the $N\ge 3$ edges emanating from $\vv$ and define $L:=\min_{i=1,\dots,N}\ell_{i}$, where $\ell_{i}:=|e_i|$ is the length of the edge $e_{i}$.
	
	Let then $\Phi_{\mu}\in H^{1}_{\mu}(S_{N})$ be the radially symmetric function on the star graph $S_N$ with $N$ half--lines whose restriction to each half--line satisfies $\Phi_\mu(x):=\phi_{\f{2\mu}{N}}(x)$, for every $x\in\R^+$. Here $\phi_{\f{2\mu}{N}}$ is the soliton at mass $\f{2\mu}N$ given by \eqref{eq:phimu}. For every $\mu>0$, let then $\delta:=\delta(\mu)$, $\kappa:=\kappa(\mu)$ be such that the function $w_{\mu}(x):=\kappa(\Phi_{\mu}(x)-\delta)_{+}$ satisfies $\|w_\mu\|_{L^2(S_N)}^2=\mu$ and it is supported on the ball $B(0,L)$ in $S_N$ of radius $L$ centered at the vertex of the star graph. Relying on the decaying properties of the solitons on the line, it is straightforward to check that $\delta\to 0$, $\kappa\to 1$ and $w_{\mu}-\Phi_{\mu}\to 0$ strongly in $H^{1}(S_{N})$ as $\mu\to +\infty$, so that $E(w_{\mu}, S_{N})-E(\Phi_{\mu}, S_{N})\to 0$ and $w_{\mu}(0)-\Phi_{\mu}(0)\to 0$ as $\mu\to +\infty$. Moreover, since $w_{\mu}$ is supported on the ball of radius $L$ centered at the vertex of $S_N$, we can think of it as a function $w_\mu\in H_\mu^1(\G)$ supported on the union of the edges $e_{i}$ emanating from the vertex $\vv$ in $\G$. Therefore, for every $\varepsilon>0$ there exists $\mu^*:=\mu^*(\varepsilon)$ so that if $\mu\geq\mu^*$
	\begin{eqnarray*}
		F_{p,q}(w_{\mu},\G)-E(\phi_{\mu},\R)&=&F_{p,q}(w_{\mu}, S_{N})-E(\phi_{\mu},\R)\\
		&=&\f{N}{2}E\left(\phi_{\f{2\mu}{N}}, \R\right)-E(\phi_{\mu},\R)-\f{1}{q}\left|\phi_{\f{2\mu}{N}}(0)\right|^{q}+\varepsilon\\
		&=&\theta_{p}\left(1-\left(\f{2}{N}\right)^{2\beta}\right)\mu^{2\beta+1}-\f{|\phi_{1}(0)|^{q}}{q}\left(\f{2}{N}\right)^{\alpha q}\mu^{\alpha q}+\varepsilon,
	\end{eqnarray*}
	where we made use of \eqref{Ephimu} and \eqref{eq:phimu}. Since $q>\f p2+1$, by Remark \ref{rem:diagpq} we have $\alpha q>2\beta+1$, so that whenever $\varepsilon$ is fixed small enough 
	\begin{equation*}
	\theta_{p}\left(1-\left(\f{2}{N}\right)^{2\beta}\right)\mu^{2\beta+1}+\varepsilon\le \f{|\phi_{1}(0)|^{q}}{q}\left(\f{2}{N}\right)^{\alpha q}\mu^{\alpha q}
	\end{equation*}
	holds for sufficiently large masses. Hence, there exists $\overline{\mu}_{p,q}>0$ so that $F_{p,q}(w_\mu,\G)\leq\ee(\mu,\R)$ for every $\mu>\overline{\mu}_{p,q}$, and by Corollary \ref{compactcor} we conclude. 
\end{proof}
\begin{proof}[Proof of Theorem \ref{3half-deg3}: existence]
	The existence result for $q<\f p2+1$ is a direct application of Proposition \ref{exsubdiag-smallmu}. On the other hand, since $\G$ satisfies Assumption (H) and it is neither the real line nor a tower of bubbles, then there is at least one vertex of degree not smaller than 3, and the existence part of the theorem for $q>\f p2+1$ follows by Proposition \ref{deg3-largemu}.
\end{proof}
Let us now focus on graphs with at least one vertex of degree 2.
\begin{proof}[Proof of Theorem \ref{3half-deg2}]
Let $\vv\in V$ be a vertex of degree $2$, $e_{1}, e_{2}$ be the two edges emanating from $\vv$ and $L:=\min\left\{|e_1|,|e_2|\right\}$.
	
For every $\mu>0$, there exists $\delta=\delta(\mu)>0$, $\kappa=\kappa(\mu)>0$ so that the function on the real line $w_\mu(x):=\kappa(\phi_{\mu}(x)-\delta)_{+}$ (where $\phi_\mu$ is the soliton at mass $\mu$ on $\R$ as in \eqref{eq:phimu}) is compactly supported on the interval $(-L,L)$ and $\|w_\mu\|_{L^2(-L,L)}^{2}=\mu$. It is straightforward to check that $\delta\to 0$, $\kappa\to 1$ and $w_\mu-\phi_{\mu}\to 0$ strongly in $H^{1}(\R)$ as $\mu\to +\infty$. This entails that $E(w_\mu, \R)-E(\phi_{\mu}, \R)\to 0$ and $w_\mu(0)-\phi_{\mu}(0)\to 0$ as $\mu\to +\infty$. Moreover, as $w_\mu$ is compactly supported on $(-L, L)$, we can think of it as a function on $\G$ supported on the union of the edges $e_{1}$ and $e_{2}$ emanating from the vertex $\vv$ of degree 2. 
	
Hence, for every $\varepsilon>0$ there exists $\mu^*=\mu^*(\varepsilon)$ such that for every $\mu\ge\mu^*$
	\begin{equation*}
	\begin{split}
	F_{p,q}(w_{\mu},\G)-E(\phi_{\mu},\R)&=E(w_\mu,(-L,L))-E(\phi_{\mu},\R)-\f{1}{q}|w_\mu(0)|^{q}\\
	&=E(w_\mu,(-L,L))-E(\phi_{\mu},\R)+\f{1}{q}\left(|\phi_{\mu}(0)|^{q}-|w_\mu(0)|^{q}\right)-\f{1}{q}|\phi_{\mu}(0)|^{q}\\
	&\le \varepsilon-\f{1}{q}|\phi_{\mu}(0)|^{q}=\varepsilon-\f{|\phi_{1}(0)|^{q}}{q}\mu^{\alpha q}.
	\end{split}
	\end{equation*}
Fixing a sufficiently small $\varepsilon$, this implies that there is $\widetilde{\mu}_{p,q}>0$ so that $F_{p,q}(w_{\mu},\G)<E(\phi_{\mu},\R)$ for every $\mu\geq\widetilde{\mu}_{p,q}$, which completes the proof by Corollary \ref{compactcor}.
\end{proof}
To conclude the analysis of the existence results, we provide the proof of Theorem \ref{2half} concerning non--compact graphs with exactly two half--lines.
\begin{proof}[Proof of Theorem \ref{2half}]
Let $\HH_{1},\HH_{2}$ be the two half--lines of the graph. Consider $u\in H^{1}_{\mu}(\G)$ defined as
\[
u(x):=\begin{cases}
\eta_m(x) & \text{if }x\in\HH_{1}\cup\HH_{2}\\
\eta_{m}(0) & \text{if }x\in\K\,,
\end{cases}
\]
where $\eta_m$ is the ground state of $F_{p,q}$ at mass $m$ on the real line as in Proposition \ref{exline}. By $\|u\|_{2}^{2}=\mu$ and \eqref{infnorm-opic} it follows that
\begin{equation*}
\mu=\|u\|^{2}_{2}=\|\eta_m\|_{L^2(\R)}^2+\ell|\eta_{m}(0)|^{2}=m+o(m)\quad \text{as}\quad \mu \to 0.
\end{equation*}
Moreover, since 
\begin{equation*}
F_{p,q}(u,\G)=E(\eta_{m},\R)+E(u,\K)-\f{n}{q}|\eta_{m}(0)|^{q}=F_{p,q}(\eta_{m},\R)-\f{\ell}{p}|\eta_{m}(0)|^{p}-\f{n-1}{q}|\eta_{m}(0)|^{q},
\end{equation*}
and 
\begin{equation*}
\begin{split}
E\left(\phi_{\mu},\R\right)=&-\theta_{p}\mu^{2\beta+1}=-\theta_p(m+\ell|\eta_m(0)|^2)^{2\beta+1}\\
=&-\theta_{p}m^{2\beta+1}-\theta_{p}(2\beta+1)\ell|\eta_{m}(0)|^{2}m^{2\beta}+o\left(|\eta_m(0)|^2m^{2\beta}\right)\quad\text{as}\quad \mu \to 0,
\end{split}
\end{equation*}
by \eqref{Fpq<Ephi} we have (recalling also \eqref{alfabeta})
\begin{equation}
\label{EuG-Ephimu}
\begin{split}
F_{p,q}(u,\G)-E\left(\phi_{\mu},\R\right)=&\,F_{p,q}(\eta_{m},\R)+\theta_{p}m^{2\beta+1}-\f{n-1}{q}|\eta_{m}(0)|^{q}\\
+&\,\left[\theta_{p}(2\beta+1)-\f{1}{p}\left(\f{\eta_{m}(0)}{m^{\alpha}}\right)^{p-2}\right]\ell|\eta_{m}(0)|^{2}m^{2\beta}+o\left(|\eta_{m}(0)|^{2}m^{2\beta}\right)\\
<& \left[\theta_{p}(2\beta+1)-\f{1}{p}\left(\f{\eta_{m}(0)}{m^{\alpha}}\right)^{p-2}\right]\ell|\eta_{m}(0)|^{2}m^{2\beta}+o\left(|\eta_{m}(0)|^{2}m^{2\beta}\right).
\end{split}
\end{equation}
Observe that 
\[
h:[0,+\infty)\to\R, \quad h(x):=\theta_{p}(2\beta+1)-\f{1}{p}x^{p-2}
\]
is  a strictly decreasing and continuous function satisfying $h(0)>0$ and $\lim_{x\to+\infty} h(x)=-\infty$, so that there is a unique $\overline{x}>0$ for which $h(\overline{x})=0$ and $h(x)<0$ if and only if  $x>\overline{x}$. By Lemma \ref{thetap-2beta} it follows that  $\overline{x}=\phi_{1}(0)$ (where as usual $\phi_{1}$ is the soliton at mass $1$ on $\R$). Since by \eqref{etamu>phimu} and \eqref{eq:phimu} we have that $\eta_{m}(0)>\phi_{m}(0)=m^\alpha\phi_1(0)$, then 
\begin{equation*}
\theta_{p}(2\beta+1)-\f{1}{p}\left(\f{|\eta_{m}(0)|}{m^{\alpha}}\right)^{p-2}<0\,.
\end{equation*} 
Therefore, coupling with \eqref{EuG-Ephimu} yields a threshold $\widehat{\mu}_{p,q}>0$ such that $F_{p,q}(u,\G)<E\left(\phi_{\mu},\R\right)$ for every $\mu\le\widehat{\mu}_{p,q}$, and by Corollary \ref{compactcor} we conclude.
\end{proof}

\section{Non--existence results: end of the proof of Theorem \ref{3half-deg3}}
\label{sec:nonex}

This section is devoted to the proof of the non--existence part of Theorem \ref{3half-deg3}. Since the argument are technically demanding, we prove two independent propositions, the first one dealing with small masses and the second one discussing the regime of large mass.
\begin{remark}
	\label{rem:omot}
	Recall that, for every graph $\G$ and every $u\in H_\mu^1(\G)$, the scaling
	\begin{equation}
	\label{scaling}
	\G\mapsto t^{-\beta}\G=:\G_t, \qquad u(\cdot)\mapsto t^{\alpha}u(t^\beta\cdot)=:u_t(\cdot)
	\end{equation}
	(with $\alpha$, $\beta$ as in \eqref{alfabeta}) preserves the quantities $\f{\|u'\|_2^2}{\mu^{2\beta+1}}$, $\f{\|u\|_p^p}{\mu^{2\beta+1}}$ for every $t>0$, that is
	\begin{equation*}
		\f{\|u'\|_{L^2(\G)}^2}{\mu^{2\beta+1}}=\f{\|u_t'\|_{L^2(\G_t)}^2}{\left(\|u_t\|_{L^2(\G_t)}^2\right)^{2\beta+1}},\qquad\f{\|u\|_{L^p(\G)}^p}{\mu^{2\beta+1}}=\f{\|u_t\|_{L^p(\G_t)}^p}{\left(\|u_t\|_{L^2(\G_t)}^2\right)^{2\beta+1}}\,,
	\end{equation*}
	so that in particular
	\begin{equation}
		\label{scalE}
		\f{E(u,\G)}{\mu^{2\beta+1}}=\f{E(u_t,\G_t)}{\left(\|u_t\|_{L^2(\G_t)}^2\right)^{2\beta+1}}\,.
	\end{equation}
\end{remark}
\begin{proposition}
	\label{noex-smallmu}
	Let $\G$ be a non--compact graph satisfying Assumption (H) with $N\ge3$ half--lines. If $q>\f{p}{2}+1$, then there exists $\underline{\mu}_{p,q}>0$ such that ground states of $F_{p,q}$ at mass $\mu$ on $\G$ do not exist for every $\mu<\underline{\mu}_{p,q}$.
\end{proposition}
\begin{proof}
	We argue by contradiction. Suppose that there exists a sequence of masses (still denoted by $\mu$, omitting the subscript of the sequence) $\mu\to 0$ so that a ground state $u_{\mu}$ of $F_{p,q}$ at mass $\mu$ exists. With no loss of generality, let $u_\mu>0$. Since $u_{\mu}$ is a ground state at mass $\mu$, by \eqref{eq:F leq E} and \eqref{Ephimu} we have
	\begin{equation}
	\label{Egs}
	E(u_{\mu},\G)\le  -\theta_{p}\mu^{2\beta+1}+ \f{1}{q}\sum_{\vv\in V}|u_\mu(\vv)|^{q}.
	\end{equation}
	Let now $\mu^{*}=\mu^{*}(p,q,1,3)$ be the critical mass associated to $F_{p,q,1}$ on the star graph $S_3$ with 3 half--lines as in Proposition \ref{exstarg} (recall that $F_{p,q,1}$ coincides with $F_{p,q}$) and fix $\widetilde{\mu}<\mu^{*}$. For every $\mu$, making use of \eqref{scaling} with $t=\f{\widetilde{\mu}}{\mu}$ and $u=u_\mu$, let then
	\begin{equation*}
	\G_{\mu}:=\left(\f{\mu}{\widetilde{\mu}}\right)^{\beta}\G,\qquad w_{\mu}(x):=\left(\f{\widetilde{\mu}}{\mu}\right)^{\alpha}u_{\mu}\left(\left(\f{\widetilde{\mu}}{\mu}\right)^{\beta}x\right)\,.
	\end{equation*}
	Clearly, $w_\mu\in H_{\widetilde{\mu}}^1(\G_{{\mu}})$ for every $\mu$. Observe that the compact core $\K_{\mu}$ of $\G_{\mu}$ has total length $|\K_{\mu}|=\widetilde{\ell}\mu^{\beta}$, where $\widetilde{\ell}=|\K|\widetilde{\mu}^{-\beta}$, so that in particular $|\K_\mu|\to0$ as $\mu\to0$.
	
	By \eqref{scalE} and \eqref{Egs} we get
	\begin{equation}
	\label{Eomo}
	E(w_{\mu},\G_{\mu})\le -\theta_{p}\widetilde{\mu}^{2\beta+1}+\left(\f{\mu}{\widetilde{\mu}}\right)^{\alpha q-(2\beta+1)}\f{1}{q}\sum_{\vv\in V}|w_{\mu}(\vv)|^{q}.
	\end{equation}
	Moreover, since $u_{\mu}$ is a ground state at sufficiently small mass $\mu$ and $q>\f p2+1$, by \eqref{kinen-mu}, \eqref{infnorm-mu} and Remark \ref{rem:omot} there exists $C>0$ (depending only on $p$) such that
	\begin{equation}
	\label{kinen-mu-omo}
	C^{-1}\widetilde{\mu}^{2\beta+1}\le\|w_{\mu}'\|_2^{2} \le C\widetilde{\mu}^{2\beta+1}
	\end{equation}
	and
	\begin{equation}
	\label{infnorm-mu-omo}
	C^{-1}\widetilde{\mu}^{\alpha}\le \|w_{\mu}\|_{\infty} \le C\widetilde{\mu}^{\alpha}.
	\end{equation}
	Let us now introduce the quantities
	\[
	\lambda_{\mu}:=\min_{\vv\in V^{+}}w_{\mu}(\vv)\quad\text{and}\quad \Lambda_{\mu}:=\max_{\K_{\mu}}w_{\mu}\,,
	\]
	(recall that $V^+$ is the set of vertices attached to at least one half--line).
	
	For the sake of clarity and to improve the readability, we divide the rest of the proof in some steps.

	{\em Step 1.} Since $\G_\mu$ satisfies Assumption $(H)$ (because $\G$ does) and it has $N\geq3$ half--lines, by Lemma \ref{rearr1}, there exists $w_{\mu}^{*}\in H^{1}_{\widetilde{\mu}}(S_{3})$ on the star graph $S_3$ with 3 half--lines such that 
	\begin{equation}
	\label{Ewmustar}
	E(w_{\mu}^{*}, S_{3})\le E(w_{\mu},\G_{\mu})
	\end{equation}
	and 
	\begin{equation}
	\label{wmustar0}
	w_{\mu}^{*}(0)=\lambda_{\mu}.
	\end{equation}  
	On the one hand, combining \eqref{Eomo} with \eqref{Ewmustar}--\eqref{wmustar0} then leads to
	\begin{equation}
	\label{Fwmustar-upbound}
	\begin{split}
	F_{p,q}(w_{\mu}^{*}, S_{3})&= E(w_{\mu}^{*}, S_{3})-\f{1}{q}|w_{\mu}^{*}(0)|^{q}\le E(w_{\mu},\G_{\mu})-\f{\lambda_{\mu}^{q}}{q}\\
	&\le -\theta_{p} \widetilde{\mu}^{2\beta+1}+\left(\f{\mu}{\widetilde{\mu}}\right)^{\alpha q-(2\beta+1)}\f1q\sum_{\vv\in V}|w_{\mu}(\vv)|^{q}-\f{\lambda_{\mu}^{q}}{q}\\
	&\le -\theta_{p} \widetilde{\mu}^{2\beta+1}+\left(\f{\mu}{\widetilde{\mu}}\right)^{\alpha q-(2\beta+1)}\f{n}{q}\Lambda_{\mu}^{q}-\f{\lambda_{\mu}^{q}}{q}.
	\end{split}
	\end{equation}
	On the other hand, since $\widetilde{\mu}<\mu^{*}=\mu^{*}(p,q,1,3)$, by Proposition \ref{exstarg} and Corollary \ref{compactcor}
	\begin{equation}
	\label{Fwmustar-lowbound}
	F_{p,q}(w_{\mu}^{*}, S_{3})>-\theta_{p} \widetilde{\mu}^{2\beta+1}.
	\end{equation}
	Coupling \eqref{Fwmustar-upbound} and \eqref{Fwmustar-lowbound} gives
	\begin{equation}
	\label{hpcase2}
	\left(\f{\mu}{\widetilde{\mu}}\right)^{\alpha q-(2\beta+1)}n \Lambda_{\mu}^{q}>\lambda_{\mu}^{q}.
	\end{equation}
	Furthermore, letting $\overline{x}\in \K_{\mu}$ realize $w_\mu(\overline{x})=\Lambda_{\mu}$ and $\overline{\vv}\in V$ be a vertex such that $w_{\mu}(\overline{\vv})=\lambda_{\mu}$, since $\K_\mu$ is connected there exists a trail $\gamma\subset \K_{\mu}$ starting at $\overline{x}$ and ending at $\overline{\vv}$. Thus 
	\begin{equation}
	\label{Mmu-Lmu}
	\Lambda_{\mu}-\lambda_{\mu}=w_{\mu}(\overline{x})-w_{\mu}(\overline{\vv})=\int_{\gamma} w'_{\mu}\,dx\le \left|\K_{\mu}\right|^{\f{1}{2}}\|w'_{\mu}\|_{L^{2}(\K_{\mu})},
	\end{equation}
	that coupled with \eqref{kinen-mu-omo}, \eqref{hpcase2} (recall also Remark \ref{rem:diagpq}) and $|\K_\mu|\to0$ as $\mu\to0$ entails
	\begin{equation} 
	\label{Mmuto0}
	\Lambda_{\mu}\to 0\quad\text{as}\quad \mu\to 0.
	\end{equation}
	
	\begin{figure}[t]
		\centering
		\subfloat{\includegraphics[width=0.49\textwidth]{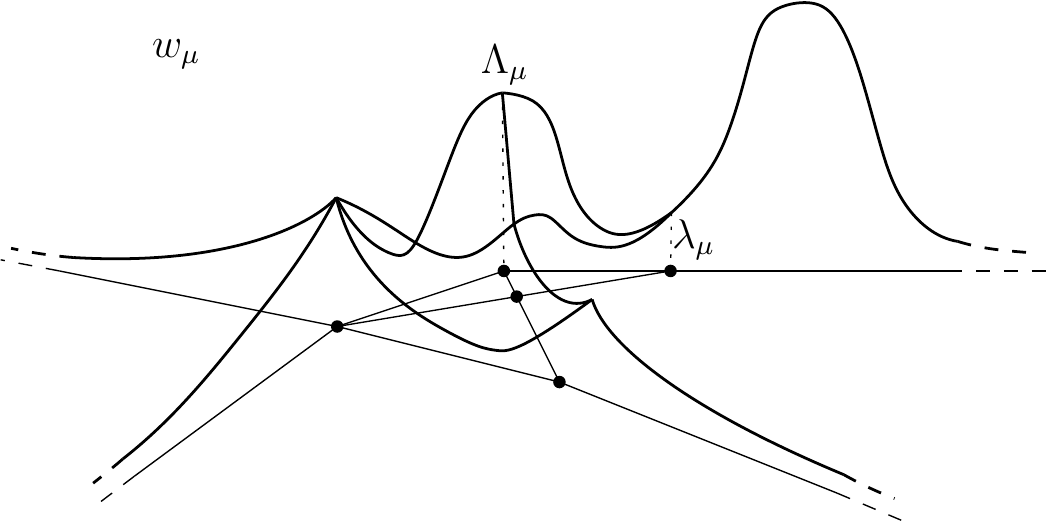}}
		\,\,
		\subfloat{\includegraphics[width=0.49\textwidth]{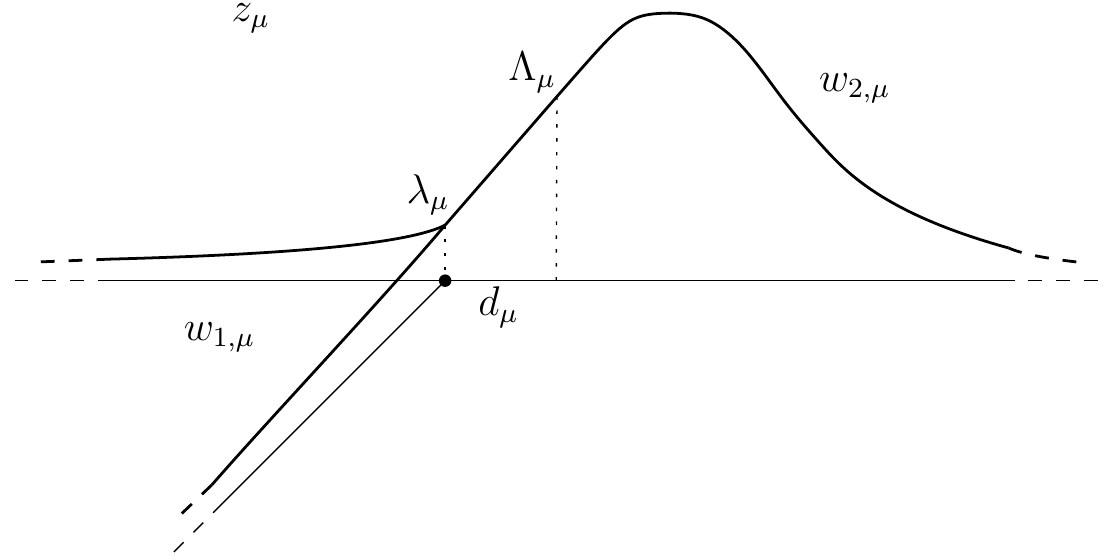}}
		\caption{The construction of the function $z_\mu$ as in Step 2 of the proof of Proposition \ref{noex-smallmu}.}
		\label{fig:zmu1}
	\end{figure}

	{\em Step 2.} By \eqref{infnorm-mu-omo} and \eqref{Mmuto0}, $\|w_{\mu}\|_{L^{\infty}(\G_{\mu})}\ge C^{-1}\widetilde{\mu}^{\alpha}>\Lambda_{\mu}=\|w_{\mu}\|_{L^{\infty}(\K)}$ for sufficiently small $\mu$. Hence, by Lemma \ref{rearr-halflines}, there exist two positive functions $w_{1,\mu},w_{2,\mu}\in H^{1}(\R^{+})$ satisfying
	\[
	\begin{split}
	&2\|w_{1,\mu}\|_{L^2(\R^+)}^2+\|w_{2,\mu}\|_{L^2(\R^+)}^2=\|w_\mu\|_{L^2(\G_{\mu}\setminus\K_{\mu})}^2\\
	&2E(w_{1,\mu},\R^+)+E(w_{2,\mu},\R^+)\leq\,E(w_\mu,\G_\mu\setminus\K_\mu)\\
	&w_{1,\mu}(0)=\lambda_\mu,\qquad w_{2,\mu}(0)=\Lambda_\mu\,.
	\end{split}
	\]
	Consider then the following construction. Denoting by $h_1,h_2,h_3$ the half--lines of the star graph $S_3$, define the function $z_{\mu}\in H^{1}(S_{3})$ as (see Figure \ref{fig:zmu1})
	\begin{equation*}
	z_{\mu}(x):=\begin{cases}
	w_{1,\mu}(x) & \text{if }x\in h_{i}\text{ for some }i=1,2\\
	\lambda_{\mu}+\f{\Lambda_{\mu}-\lambda_{\mu}}{d_{\mu}}x & \text{if }x\in[0,d_{\mu})\cap h_3,\\
	w_{2,\mu}(x-d_{\mu}) & \text{if }x\in[d_\mu,+\infty)\cap h_3,
	\end{cases}
	\end{equation*}
	where $d_{\mu}:=d\mu^{\beta}$, $d>0$ fixed, will be properly chosen later. By construction, $z_{\mu}$ is decreasing on $h_{1}$ and $h_{2}$, while on $h_{3}$ it is increasing (with a linear part at the beginning) until a maximum point and then decreasing from this point on. Moreover, as $z_\mu(x)\leq \Lambda_\mu$ for every $x\in[0,d_\mu)\cap h_3$,
	\begin{equation*}
	\begin{split}
	\|z_{\mu}\|_{L^{2}(S_{3})}^{2}&=2\|w_{1,\mu}\|_{L^2(\R^+)}^{2}+\|w_{2,\mu}\|_{L^2(\R^+)}^{2}+\int_{[0,d_\mu)\cap h_3}|z_{\mu}|^{2}\,dx\\
	&\le \|w_{\mu}\|_{L^2(\G_{\mu}\setminus \K_{\mu})}^{2}+d_{\mu}\Lambda_{\mu}^{2}=\widetilde{\mu}-\|w_{\mu}\|_{L^2(\K_\mu)}^{2}+d\mu^{\beta}\Lambda_{\mu}^{2}\\
	&\le \widetilde{\mu}-\widetilde{\ell}\mu^{\beta}\left(\min_{\K_{\mu}} w_{\mu}\right)^{2}+d\mu^{\beta}\Lambda_{\mu}^{2},
	\end{split}
	\end{equation*}
	so that, since by \eqref{hpcase2}
	\[
	\min_{\K_{\mu}} w_{\mu}\le \lambda_{\mu}=o\left(\Lambda_{\mu}\right)\quad \text{as}\quad \mu\to 0, 
	\]
	we obtain
	\begin{equation}
	\label{mass-est1}
	\|z_{\mu}\|_{L^{2}(S_{3})}^{2}\le \widetilde{\mu}+d\mu^{\beta}\Lambda_{\mu}^{2}+o\left(\mu^{\beta}\Lambda_{\mu}^{2}\right).
	\end{equation}
	Let us now estimate $F_{p,q}(z_{\mu}, S_{3})$. On the one hand,  if $\mu$ is sufficiently small, then \eqref{mass-est1} and $\widetilde{\mu}<\mu^*$ ensure $\|z_{\mu}\|_{L^{2}(S_{3})}^{2}<\mu^{*}$, so that by Proposition \ref{exstarg} with $\tau=1$ and Corollary \ref{compactcor} 
	\begin{equation*}
	F_{p,q}(z_{\mu}, S_{3})>\F_{p,q}\left(\|z_{\mu}\|_{2}^{2},S_{3}\right)=-\theta_{p}\left(\|z_{\mu}\|_{2}^{2}\right)^{2\beta+1},
	\end{equation*}
	and by \eqref{mass-est1} again
	\begin{equation}
	\label{estFfrombe}
	\begin{split}
	F_{p,q}(z_{\mu}, S_{3})&\ge -\theta_{p}\left(\widetilde{\mu}+d\mu^{\beta}\Lambda_{\mu}^{2}\right)^{2\beta+1}+o\left(\mu^{\beta}\Lambda_{\mu}^{2}\right)\\
	&=-\theta_{p}\widetilde{\mu}^{2\beta+1}-d(2\beta+1)\widetilde{\mu}^{2\beta}\mu^{\beta}\Lambda_{\mu}^{2}+o\left(\mu^{\beta}\Lambda_{\mu}^{2}\right)\quad\text{as}\quad \mu\to 0.
	\end{split}
	\end{equation}
	On the other hand, exploiting the construction of $z_\mu$ and the properties of $w_{1,\mu},w_{2,\mu}$,
	\begin{equation}
	\label{estEfromab}
	\begin{split}
	E(z_{\mu},S_{3})&=E\left(z_{\mu},S_{3}\setminus \left([0,d_{\mu})\cap h_3\right)\right)+E\left(z_{\mu}, [0,d_{\mu})\cap h_3\right)\\
	&\le E(w_{\mu}, \G_{\mu}\setminus\K_{\mu})+\f{\left(\Lambda_{\mu}-\lambda_{\mu}\right)^{2}}{2d\mu^{\beta}}-\f{d\mu^{\beta}}{p}\lambda_{\mu}^{p}\\
	&\le E(w_{\mu}, \G_{\mu}\setminus\K_{\mu})+\f{\Lambda_{\mu}^{2}}{2d\mu^{\beta}}+o\left(\f{\Lambda_{\mu}^{2}}{\mu^{\beta}}\right),
	\end{split}
	\end{equation}
	the last inequality relying  also on \eqref{Mmuto0}. Coupling \eqref{estEfromab} with \eqref{Eomo} then yields
	\begin{equation}
	\label{estFfromab}
	\begin{split}
	F_{p,q}(z_{\mu}, S_{3})&\le  E(w_{\mu}, \G_{\mu}\setminus\K_{\mu})+\f{\Lambda_{\mu}^{2}}{2d\mu^{\beta}}-\f{\lambda_{\mu}^{q}}{q}+o\left(\f{\Lambda_{\mu}^{2}}{\mu^{\beta}}\right)\\
	&=  E(w_{\mu},\G_{\mu})-E(w_{\mu},\K_{\mu})+\f{\Lambda_{\mu}^{2}}{2d\mu^{\beta}}-\f{\lambda_{\mu}^{q}}{q}+o\left(\f{\Lambda_{\mu}^{2}}{\mu^{\beta}}\right)\\
	&\le  -\theta_{p}\widetilde{\mu}^{2\beta+1}+\left(\f{\mu}{\widetilde{\mu}}\right)^{\alpha q-(2\beta+1)}\f{1}{q}\sum_{\vv\in V}|w_{\mu}(\vv)|^{q}-E(w_{\mu},\K_{\mu})\\
	&\quad+\f{\Lambda_{\mu}^{2}}{2d\mu^{\beta}}-\f{\lambda_{\mu}^{q}}{q}+o\left(\f{\Lambda_{\mu}^{2}}{\mu^{\beta}}\right)\\
	&\le  -\theta_{p}\widetilde{\mu}^{2\beta+1}+C\mu^{\alpha q-(2\beta+1)}\Lambda_{\mu}^{q}-E(w_{\mu},\K_{\mu})+\f{\Lambda_{\mu}^{2}}{2d\mu^{\beta}}+o\left(\f{\Lambda_{\mu}^{2}}{\mu^{\beta}}\right)
	\end{split}
	\end{equation}
	for some constant $C>0$.
	
	{\em Step 3.} By comparing \eqref{estFfrombe} and \eqref{estFfromab} and noting that $\mu^{\beta}\Lambda_{\mu}^{2}=o\left(\f{\Lambda_{\mu}^{2}}{\mu^{\beta}}\right)$ as $\mu\to 0$, we obtain
	\begin{equation}
	\label{lastineq}
	C\mu^{\alpha q-(2\beta+1)}\Lambda_{\mu}^{q}\ge E(w_{\mu},\K_{\mu})-\f{\Lambda_{\mu}^{2}}{2d\mu^{\beta}}+o\left(\f{\Lambda_{\mu}^{2}}{\mu^{\beta}}\right).
	\end{equation}
	Let us now estimate the term $E(w_{\mu},\K_{\mu})$. By \eqref{Mmu-Lmu}, 
	\begin{equation*}
	\|w'_{\mu}\|^{2}_{L^{2}(\K_{\mu})}\ge \f{(\Lambda_{\mu}-\lambda_{\mu})^{2}}{|\K_{\mu}|}=\f{\Lambda_{\mu}^{2}}{\widetilde{\ell}\mu^{\beta}}+o\left(\f{\Lambda_{\mu}^{2}}{\mu^{\beta}}\right)
	\end{equation*}
	and
	\begin{equation*}
	\|w_{\mu}\|^{p}_{L^{p}(\K_{\mu})}\le \widetilde{\ell}\mu^{\beta}\Lambda_{\mu}^{p}\le C'\mu^{\beta}\|w'_{\mu}\|_{L^2(\K_\mu)}^{\f p2}=o\left(\|w'_{\mu}\|^{2}_{L^{2}(\K_{\mu})}\right)
	\end{equation*}
	for a suitable $C'>0$. Therefore,
	\begin{equation*}
	E(w_{\mu},\K_{\mu})=\f{1}{2}\|w'_{\mu}\|^{2}_{L^{2}(\K_{\mu})}+o\left(\|w'_{\mu}\|^{2}_{L^{2}(\K_{\mu})}\right)\ge \f{\Lambda_{\mu}^{2}}{2\widetilde{\ell}\mu^{\beta}}+o\left(\f{\Lambda_{\mu}^{2}}{\mu^{\beta}}\right)
	\end{equation*}
	and plugging into \eqref{lastineq} we get
	\begin{equation*}
	C\mu^{\alpha q-(2\beta+1)}\Lambda_{\mu}^{q}\ge \f{1}{2}\left(\f{1}{\widetilde{\ell}}-\f{1}{d}\right)\f{\Lambda_{\mu}^{2}}{\mu^{\beta}}+o\left(\f{\Lambda_{\mu}^{2}}{\mu^{\beta}}\right)\,.
	\end{equation*}
	Choosing for instance $d=2\widetilde{\ell}$ then leads to
	\begin{equation*}
	\mu^{\alpha q-(\beta+1)}\Lambda_{\mu}^{q-2}\ge \f{1}{4C\widetilde{\ell}}\quad\text{as }\mu\to0,
	\end{equation*}
	that is a contradiction in light of Remark \ref{rem:diagpq}, thus concluding the proof.
\end{proof}

\begin{proposition}
	\label{noex-largemu}
	Let $\G$ be a non--compact graph satisfying Assumption (H) with no vertex of degree smaller than $3$. If $q<\f p2+1$,  then there exists $\overline{\mu}_{p,q}>0$ such that ground states of $F_{p,q}$ at mass $\mu$ on $\G$ do not exist for every $\mu>\overline{\mu}_{p,q}$.
\end{proposition}
\begin{proof}
	We argue by contradiction. Suppose that there exists a sequence of masses (still denoted by $\mu$, omitting the subscript of the sequence) $\mu\to +\infty$ so that a ground state $u_{\mu}$ of $F_{p,q}$ at mass $\mu$ exists. With no loss of generality let $u_\mu>0$. Since $u_{\mu}$ is a ground state at mass $\mu$, \eqref{Egs} holds.
	
	Let $\mu^{*}=\mu^{*}(p,q,1,3)$ be the critical mass associated to $F_{p,q,1}$ on the star graph $S_3$ with 3 half--lines as in Proposition \ref{exstarg} (recall that $F_{p,q,1}$ coincides with $F_{p,q}$) and fix $\widetilde{\mu}>\mu^{*}$. For every $\mu$, making use of \eqref{scaling} with $t=\f{\widetilde{\mu}}{\mu}$ and $u=u_\mu$, let then
	\begin{equation*}
	\G_{\mu}:=\left(\f{\mu}{\widetilde{\mu}}\right)^{\beta}\G,\qquad w_{\mu}(x):=\left(\f{\widetilde{\mu}}{\mu}\right)^{\alpha}u_{\mu}\left(\left(\f{\widetilde{\mu}}{\mu}\right)^{\beta}x\right)\,,
	\end{equation*}
	so that $\K_{\mu}$ has total length $|\K_{\mu}|=\widetilde{\ell}\mu^{\beta}\to +\infty$ as $\mu\to +\infty$ (with $\widetilde{\ell}:=|\K|\widetilde{\mu}^{-\beta}$) and, for every $\mu$, we have $w_\mu\in H_{\widetilde{\mu}}^1(\G_{{\mu}})$ and, by \eqref{scalE} and \eqref{Egs},
	\begin{equation}
	\label{Eomo2}
	E(w_{\mu},\G_{\mu})\le -\theta_{p}\widetilde{\mu}^{2\beta+1}+\left(\f{\widetilde{\mu}}{\mu}\right)^{2\beta+1-\alpha q}\f{1}{q}\sum_{\vv\in V}|w_{\mu}(\vv)|^{q}.
	\end{equation}
	The remainder of the proof is divided in three steps.
	
	\begin{figure}[t]
		\centering
		\subfloat{\includegraphics[width=0.48\textwidth]{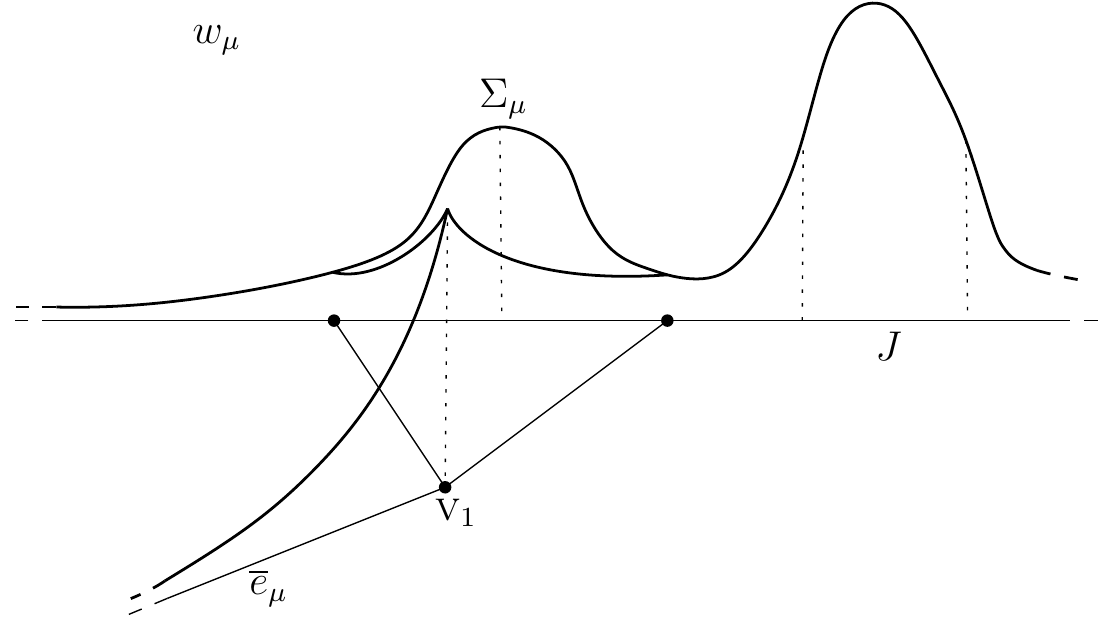}}
		\,\,
		\subfloat{\includegraphics[width=0.48\textwidth]{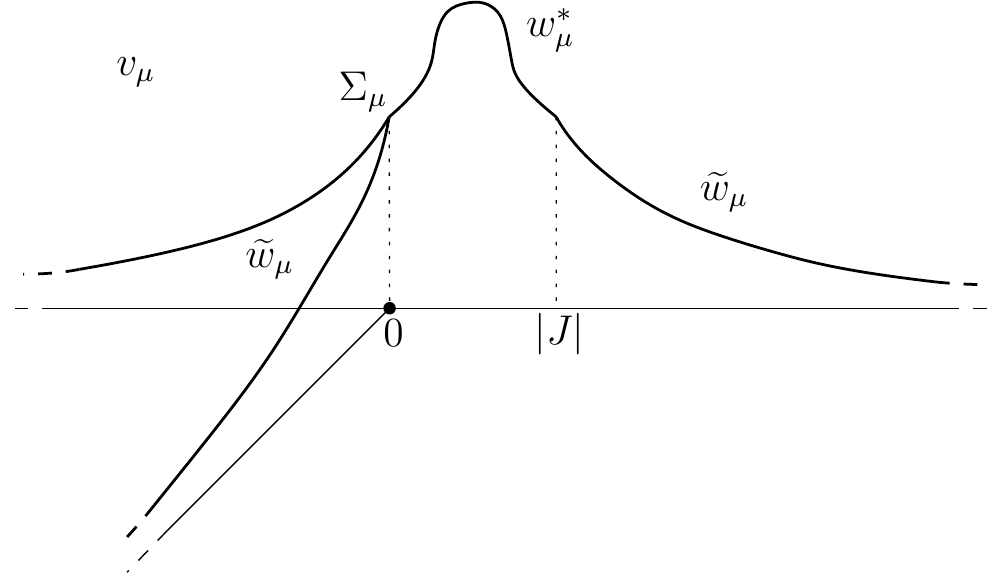}}
		\caption{The construction of the function $v_\mu$ as in Step 1 of the proof of Proposition \ref{noex-largemu}.}
		\label{fig:vmu}
	\end{figure}
	
	{\em Step 1.} Let $\overline{x}_\mu\in\G_\mu$ be such that $w_\mu(\overline{x}_\mu)=\|w_\mu\|_{\infty}$. By Assumption (H), there exists a trail $\gamma_{\mu}\subset\G_\mu$ running through $\overline{x}_\mu$ and exactly 2 half--lines of $\G_\mu$. Let then $\vv_1\in V$ be a vertex of $\G_\mu$ satisfying $w_{\mu}(\vv_{1})=\max_{\vv\in V}w_{\mu}(\vv)$. Since by hypothesis every vertex of $\G_\mu$ has degree greater than or equal to $3$, there is at least one edge, say $\overline{e}_{\mu}$, emanating from $\vv_{1}$ that does not belong to $\gamma_{\mu}$. 
	Set
	\[
	\sigma_{\mu}:=\inf_{\overline{e}_{\mu}}w_{\mu}\quad\text{and}\quad \Sigma_{\mu}:=\max_{\overline{e}_{\mu}}w_{\mu}\,.
	\]
	Let us first show that $\overline{e}_\mu$ must be a bounded edge of $\G_\mu$ for every $\mu$ large enough. Assume on the contrary that there exists a subsequence (not renamed) $\mu\to +\infty$ along which $\overline{e}_{\mu}$ is a half--line. Letting $J:=\{x\in \G_{\mu}\,:\,w_{\mu}(x)>\Sigma_{\mu}\}$, then $w_{\mu}(J)$ is connected and, since $w_\mu(\gamma_\mu)=(0,\|w_\mu\|_\infty]$ and $\gamma_\mu$ is a trail, every $t\in w_{\mu}(J)$ is attained at least twice, except possibly $\|w_{\mu}\|_{\infty}$. Let then $w^{*}_{\mu}\in H^{1}\left(-\f{|J|}{2},\f{|J|}{2}\right)$ be the symmetric rearrangement of $w_{\mu|J}$ on $\left(-\f{|J|}2,\f{|J|}2\right)$, which satisfies
	\begin{equation*}
	\|w_{\mu}'\|_{L^2(J)}\geq\|(w^{*}_{\mu})'\|_{L^2\left(-\f{|J|}{2},\f{|J|}{2}\right)}\,,\quad\|w_{\mu}\|_{L^r(J)}=\|w^{*}_{\mu}\|_{L^r\left(-\f{|J|}{2},\f{|J|}{2}\right)}\quad\forall r\geq1.
	\end{equation*}
	Moreover, $w_\mu\left(\G_{\mu}\setminus J\right)$ is connected too and every $t\in w_{\mu}(\G_{\mu}\setminus J)$ is attained at least three times (at least twice on $\gamma_\mu$ and at least once on $\overline{e}_\mu$). Thus, taking the symmetric rearrangement $\widetilde{w_{\mu}}\in H^{1}(S_{3})$ of $w_{\mu|(\G_{\mu}\setminus J)}$ on the star graph $S_3$ with 3 half--lines (see \cite[Appendix A]{acfn_jde}) we have 
	\begin{equation*}
	\|w_{\mu}'\|_{L^2(\G_{\mu}\setminus J)}\geq\|\widetilde{w_{\mu}}'\|_{L^2\left(S_{3}\right)}\,,\quad\|w_{\mu}\|_{L^r(\G_{\mu}\setminus J)}=\|\widetilde{w_{\mu}}\|_{L^r\left(S_{3}\right)}\quad\forall r\geq1.
	\end{equation*}
	Denoting as usual by $h_1,h_2,h_3$ the half--lines of $S_3$, we then define $v_{\mu}\in H^{1}_{\widetilde{\mu}}(S_{3})$ as (see Figure \ref{fig:vmu})
	\begin{equation*}
	v_{\mu}(x):=
	\begin{cases}
	\widetilde{w_{\mu}}(x)\quad &x\in h_1\cup h_2\\
	w^{*}_{\mu}\left(x-\f{|J|}{2}\right)\quad &x\in [0,|J|)\cap h_3\\
	\widetilde{w_{\mu}}(x-|J|)\quad &x\in [|J|, +\infty)\cap h_3,
	\end{cases}
	\end{equation*}
	so that, by \eqref{Eomo2} and for $\mu$ sufficiently large
	\begin{equation*}
	\begin{split}
	F_{p,q}(v_{\mu}, S_{3})&=E(v_{\mu}, S_{3})-\f{1}{q}|v_{\mu}(0)|^{q}\le E(w_{\mu}, \G_{\mu})-\f{\Sigma_{\mu}^{q}}{q}\\
	&\le -\theta_{p}\widetilde{\mu}^{2\beta+1}+\left(\f{\widetilde{\mu}}{\mu}\right)^{2\beta+1-\alpha q}\f{1}{q}\sum_{\vv\in V}|w_{\mu}(\vv)|^{q}-\f{\Sigma_{\mu}^{q}}{q}\\
	&\le -\theta_{p}\widetilde{\mu}^{2\beta+1}-\f{\Sigma_{\mu}^{q}}{q}\left(1-n\left(\f{\widetilde{\mu}}{\mu}\right)^{2\beta+1-\alpha q}\right)<-\theta_{p}\widetilde{\mu}^{2\beta+1},
	\end{split}
	\end{equation*}
	the last inequality being a direct consequence of Remark \ref{rem:diagpq} and $\mu\to+\infty$. By Corollary \ref{compactcor} this implies existence of ground states of $F_{p,q}$ on $S_3$ at mass $\widetilde{\mu}>\mu^*=\mu^*(p,q,1,3)$, which is impossible by Proposition \ref{exstarg}. Hence, for every $\mu$ large enough $\overline{e}_\mu$ is a bounded edge.
	
	\begin{figure}[t]
		\centering
		\subfloat{\includegraphics[width=0.48\textwidth]{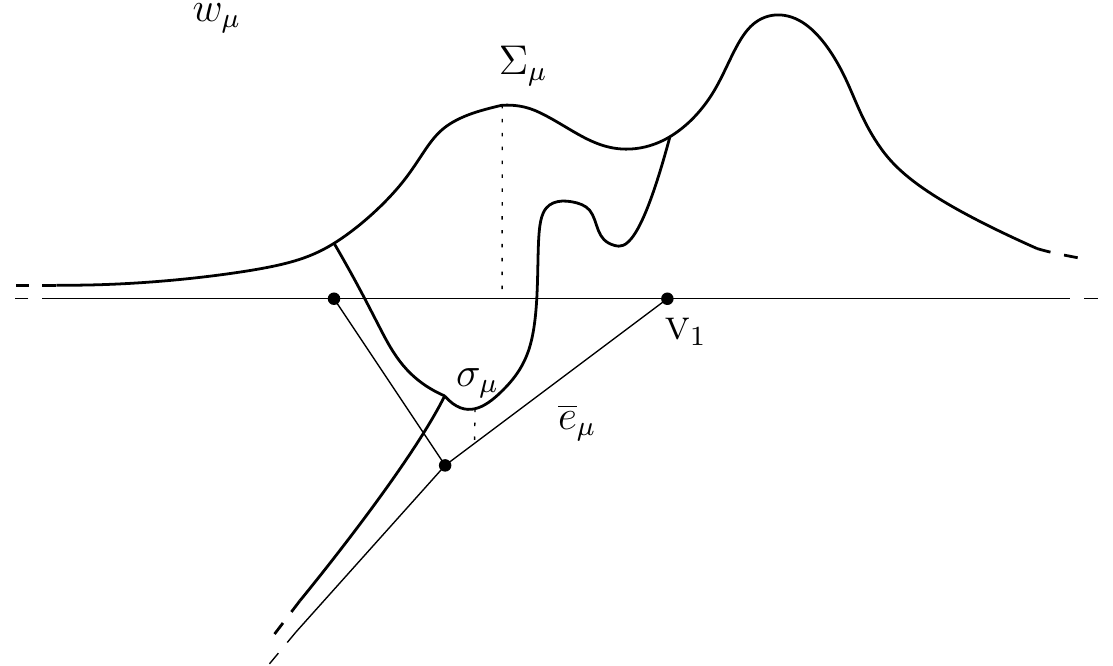}}
		\,\,
		\subfloat{\includegraphics[width=0.48\textwidth]{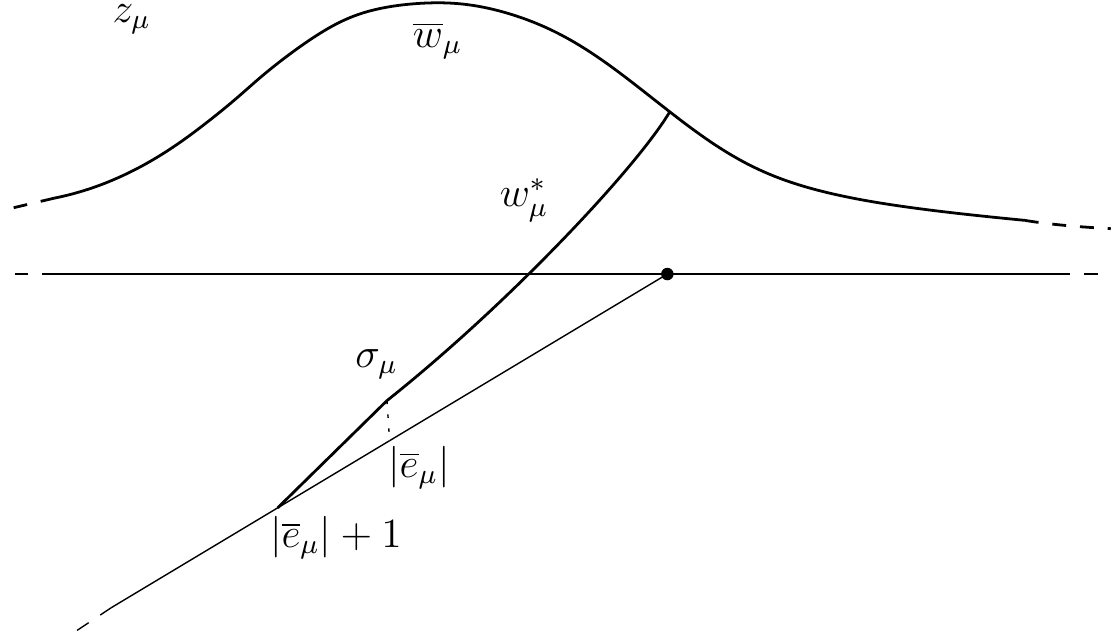}}
		\caption{The construction of the function $z_\mu$ as in Step 2 of the proof of Proposition \ref{noex-largemu}.}
		\label{fig:zmudim2}
	\end{figure}
	
	{\em Step 2.} Since $\overline{e}_\mu$ is a bounded edge, then by definition of $\G_\mu$ there exist two constants $c_1,c_2>0$ so that
	\begin{equation*}
	c_{1}\mu^{\beta}\le|\overline{e}_{\mu}|\le c_{2}\mu^{\beta}\,.
	\end{equation*} 
	Moreover, as $\widetilde{\mu}=\|w_\mu\|_{L^2(\G_\mu)}^2\ge \|w_{\mu}\|_{L^{2}(\overline{e}_{\mu})}^{2}\ge  c_{1}\mu^{\beta}\sigma_{\mu}^{2}$, it follows that $\sigma_{\mu}\to 0$ as $\mu\to +\infty$. 
	
	Consider then the following construction. Recall that $w_{\mu}(\G_{\mu}\setminus \overline{e}_{\mu})$ is connected and, since $\gamma_\mu\subset\G_{\mu}\setminus \overline{e}_{\mu}$ and $w_\mu(\gamma_\mu)=(0,\|w_\mu\|_\infty]$, every value $t\in w_{\mu}(\G_{\mu}\setminus \overline{e}_{\mu})$ is attained at least twice, except possibly $\|w_{\mu}\|_{\infty}$. Letting $\overline{w}_\mu\in H^{1}(\R)$ be the symmetric rearrangement on $\R$ of the restriction of $w_{\mu}$ to $\G_{\mu}\setminus\overline{e}_{\mu}$, we have  
	\[
	\|w_{\mu}'\|_{L^2(\G_{\mu}\setminus \overline{e}_{\mu})}\geq\|\overline{w}_\mu'\|_{L^2(\R)}\,,\quad\|w_{\mu}\|_{L^r(\G_{\mu}\setminus \overline{e}_{\mu})}=\|\overline{w}_\mu\|_{L^r(\R)}\quad\forall r\geq1.
	\]
	Moreover, denoting by $w_{\mu}^*\in H^{1}\left(0,|\overline{e}_{\mu}|\right)$ the decreasing rearrangement of the restriction of $w_{\mu}$ to $\overline{e}_{\mu}$, we also get 
	\[
	\|w_{\mu}'\|_{L^2(\overline{e}_{\mu})}\geq\|(w_{\mu}^*)'\|_{L^2(\overline{e}_{\mu})}\,,\quad\|w_{\mu}\|_{L^r(\G_{\mu}\setminus \overline{e}_{\mu})}=\|w_{\mu}^*\|_{L^r(\overline{e}_{\mu})}\quad\forall r\geq1
	\]
	and 
	\begin{equation*}
	w_{\mu}^*(0)=\Sigma_{\mu}\,,\quad w_{\mu}^*(|\overline{e}_{\mu}|)=\sigma_{\mu}.
	\end{equation*}
	Observe that, since $w_{\mu}^*(0)\le \|\overline{w}_\mu\|_{L^{\infty}(\R)}$, there exists $\xi\in \R$ such that $\overline{w}_\mu(\xi)=w_{\mu}^*(0)$.  Hence, letting as usual $h_1,h_2,h_3$ be the half--lines of $S_3$, we set $z_{\mu}\in H^{1}(S_{3})$ as (see Figure \ref{fig:zmudim2})
	\begin{equation*}
	z_{\mu}(x):=
	\begin{cases}
	\overline{w}_{\mu}(x+\xi)\quad &x\in h_{1}\cup h_{2}\\
	w_{\mu}^*(x)\quad &x\in [0,|\overline{e}_{\mu}|)\cap h_3\\
	\sigma_{\mu}(|\overline{e}_{\mu}|+1-x)\quad &x\in [|\overline{e}_{\mu}|, |\overline{e}_{\mu}|+1)\cap h_3\\
	0\quad &\text{elsewhere}\,.
	\end{cases}
	\end{equation*}
	Note that $z_{\mu}(0)=\Sigma_{\mu}$ and 
	\begin{equation*}
	\begin{split}
	\|z_{\mu}\|_{L^{2}(S_{3})}^{2}&=\int_{S_{3}\setminus\left([|\overline{e}_{\mu}|,|\overline{e}_{\mu}|+1)\cap h_3\right)} |z_{\mu}|^{2}\,dx+\int_{[|\overline{e}_{\mu}|,|\overline{e}_{\mu}|+1)\cap h_3}|z_{\mu}|^{2}\,dx\\
	&=\|w_{\mu}\|_{L^2(\G_\mu)}^{2}+\int_{0}^{1} |\sigma_{\mu}x|^{2}\,dx = \widetilde{\mu}+\f{\sigma_{\mu}^{2}}{3}.
	\end{split}
	\end{equation*}
	
	\noindent On the one hand, since $\|z_{\mu}\|_{L^{2}(S_{3})}^{2}\ge \widetilde{\mu}>\mu^{*}=\mu^*(p,q,1,3)$, by Proposition \ref{exstarg} and Corollary \ref{compactcor}
	\begin{equation}
	\label{Fgmulowbound}
	\begin{split}
	F_{p,q}(z_{\mu}, S_{3})&\ge -\theta_{p}\left(\|z_{\mu}\|_{L^{2}(S_{3})}^{2}\right)^{2\beta+1}\ge -\theta_{p}\left(\widetilde{\mu}+\f{\sigma_{\mu}^{2}}{3}\right)^{2\beta+1}\\
	&=-\theta_{p}\widetilde{\mu}^{2\beta+1}-\f{\theta_{p}}{3}(2\beta+1)\widetilde{\mu}^{2\beta}\sigma_{\mu}^{2}+o\left(\sigma_{\mu}^{2}\right).
	\end{split}
	\end{equation}
	On the other hand, 
	\begin{equation*}
	\begin{split}
	E(z_{\mu},S_{3}) &\le E(w_{\mu}, \G_{\mu})+\f{1}{2}\int_{0}^{1}|(\sigma_{\mu}x)'|^{2}\,dx-\f{1}{p}\int_{0}^{1}|\sigma_{\mu}x|^{p}\,dx\\
	&=E(w_{\mu}, \G_{\mu})+\f{\sigma_{\mu}^{2}}{2}-\f{\sigma_{\mu}^{p}}{p(p+1)}=E(w_{\mu}, \G_{\mu})+\f{\sigma_{\mu}^{2}}{2}+o\left(\sigma_{\mu}^{2}\right),
	\end{split}
	\end{equation*}
	so that, recalling \eqref{Eomo2}, the fact that $\Sigma_{\mu}\geq w_\mu(\vv_1)=\max_{\vv\in V}w_\mu(\vv)$ by construction, and that $2\beta+1-\alpha q>0$ by Remark \ref{rem:diagpq}
	\begin{equation}
	\label{Fgmuupbound}
	\begin{split}
	F_{p,q}(z_{\mu}, S_{3})&=E(z_{\mu}, S_{3})-\f{1}{q}|z_{\mu}(0)|^{q} \le E(w_{\mu}, \G_{\mu})+\f{\sigma_{\mu}^{2}}{2}+o\left(\sigma_{\mu}^{2}\right) -\f{\Sigma_{\mu}^{q}}{q}\\
	&\le -\theta_{p}\widetilde{\mu}^{2\beta+1}+\f{n}{q}\left(\f{\widetilde{\mu}}{\mu}\right)^{2\beta+1-\alpha q}\Sigma_{\mu}^{q}+\f{\sigma_{\mu}^{2}}{2}+o\left(\sigma_{\mu}^{2}\right)-\f{\Sigma_{\mu}^{q}}{q}\\
	&=-\theta_{p}\widetilde{\mu}^{2\beta+1}-\f{\Sigma_{\mu}^{q}}{q}+o\left(\Sigma_{\mu}^{q}\right)+\f{\sigma_{\mu}^{2}}{2}+o\left(\sigma_{\mu}^{2}\right).
	\end{split}
	\end{equation}
	Comparing \eqref{Fgmulowbound} and \eqref{Fgmuupbound}, we get
	\begin{equation}
	\label{hp1case2}
	\f{\Sigma_{\mu}^{q}}{q}+o\left(\Sigma_{\mu}^{q}\right)\le \left(\f{1}{2}+\f{\theta_{p}}{3}(2\beta+1)\widetilde{\mu}^{2\beta}\right)\sigma_{\mu}^{2}+o\left(\sigma_{\mu}^{2}\right)\quad \text{as}\quad \mu \to +\infty\,.
	\end{equation}
	
	{\em Step 3. }Since $\sigma_\mu\to0$ as $\mu\to+\infty$, by \eqref{hp1case2} also $\Sigma_\mu\to0$ as $\mu\to+\infty$ and there exists $C>0$ so that
	\begin{equation*}
	\begin{split}
	&\|w_{\mu}\|^{2}_{L^{2}(\overline{e}_{\mu})}\ge c_{1}\mu^{\beta}\sigma_{\mu}^{2}\ge C\mu^{\beta}\Sigma_{\mu}^{q}\\
	&\|w_{\mu}\|^{2}_{L^{2}(\G_{\mu}\setminus\overline{e}_{\mu})}\le \widetilde{\mu}-C\mu^{\beta}\Sigma_{\mu}^{q}\,.
	\end{split}
	\end{equation*}
	Moreover, 
	\begin{equation*}
	\begin{split}
	E(w_{\mu}, \G_{\mu}\setminus \overline{e}_{\mu})&\ge E(z_{\mu},h_{1}\cup h_{2})\ge \Eps\left(\widetilde{\mu}-C\mu^{\beta}\Sigma_{\mu}^{q},\R\right)=-\theta_{p}\left(\widetilde{\mu}-C\mu^{\beta}\Sigma_{\mu}^{q}\right)^{2\beta+1}\\
	&=-\theta_{p}\widetilde{\mu}^{2\beta+1}+\theta_{p}\widetilde{\mu}^{2\beta}(2\beta+1)C\mu^{\beta}\Sigma_{\mu}^{q}+o\left(\mu^{\beta}\Sigma_{\mu}^{q}\right)
	\end{split}
	\end{equation*}
	and
	\begin{equation*}
	E(w_{\mu}, \overline{e}_{\mu})\ge-\f{\|w_{\mu}\|^{p}_{L^{p}(\overline{e}_{\mu})}}{p}\ge-\f{c_{2}\mu^{\beta}\Sigma_{\mu}^{p}}{p}\,.
	\end{equation*}
	Hence, since $q<\f{p}{2}+1<p$ and $\Sigma_{\mu}\to 0$ as $\mu\to +\infty$
	\begin{equation}
	\label{EwmuGmu}
	\begin{split}
	E(w_{\mu}, \G_{\mu})&\ge -\theta_{p}\widetilde{\mu}^{2\beta+1}+\theta_{p}\widetilde{\mu}^{2\beta}(2\beta+1)C\mu^{\beta}\Sigma_{\mu}^{q}-\f{c_{2}\mu^{\beta}\Sigma_{\mu}^{p}}{p}+o\left(\mu^{\beta}\Sigma_{\mu}^{q}\right)\\
	&=-\theta_{p}\widetilde{\mu}^{2\beta+1}+\theta_{p}\widetilde{\mu}^{2\beta}(2\beta+1)C\mu^{\beta}\Sigma_{\mu}^{q}+o\left(\mu^{\beta}\Sigma_{\mu}^{q}\right).
	\end{split}
	\end{equation}
	By \eqref{Eomo2} and \eqref{EwmuGmu}, it follows that
	\begin{equation*}
	\theta_{p}\widetilde{\mu}^{2\beta}(2\beta+1)C\mu^{\beta}\Sigma_{\mu}^{q}+o\left(\mu^{\beta}\Sigma_{\mu}^{q}\right)\le \f{1}{q} \left(\f{\widetilde{\mu}}{\mu}\right)^{2\beta+1-\alpha q} \sum_{\vv\in V}|w_{\mu}(\vv)|^{q} \le \f{n}{q} \left(\f{\widetilde{\mu}}{\mu}\right)^{2\beta+1-\alpha q}\Sigma_{\mu}^{q}, 
	\end{equation*}
	which entails
	\begin{equation}
	\label{boundmu}
	1+o\left(1\right)\le C'\mu^{-(3\beta+1)+\alpha q} \quad\text{as}\quad \mu\to +\infty\,.
	\end{equation}
	By Remark \ref{rem:diagpq}, since $q<\f p2+1$, then $3\beta+1>2\beta+1>\alpha q$, and \eqref{boundmu} provides the contradiction we seek.
\end{proof}
\begin{proof}[End of the proof of Theorem \ref{3half-deg3}: non--existence.]
	It is the content of Propositions \ref{noex-smallmu}--\ref{noex-largemu}.
\end{proof}



\section{Metric results: proof of Theorems \ref{thm:metric1}--\ref{thm:metric2}--\ref{thm:metric3}}
\label{sec:metr}

This section provides the proof of the metric results of the paper. Before proving our main theorems in this context, let us state the following straightforward lemma.
\begin{lemma}
\label{exRcompactsupp}
Let $p\in(2,6), q\in(2,4)$. Then there exists $\overline{\mu}>0$ (depending on $p,q$) such that for every $\mu\ge\overline{\mu}$ there is a function $f_{\mu}\in H^{1}_{\mu}(\R)$ compactly supported in $[-1,1]$ and satisfying 
\begin{equation*}
F_{p,q}(f_{\mu}, \R)<E(\phi_{\mu},\R) .
\end{equation*}
\end{lemma}

\begin{proof}

\noindent Let $\phi_{\mu}$ be the soliton at mass $\mu$ on $\R$ as in \eqref{eq:phimu} and choose $\delta=\delta(\mu)$ and $\kappa=\kappa(\mu)$ such that the function $f_\mu(x):=\kappa(\phi_{\mu}(x)-\delta)_{+}$ is compactly supported in $[-1,1]$ and $\|f_\mu\|_{2}^{2}=\mu$. It is immediate to check that $\delta\to 0$, $\kappa\to 1$ and $f_\mu-\phi_{\mu}\to 0$ strongly in $H^{1}(\R)$ as $\mu\to +\infty$. This entails that $E(f_\mu, \R)-E(\phi_{\mu}, \R)\to 0$ and $f_\mu(0)-\phi_{\mu}(0)\to 0$ as $\mu\to +\infty$. Hence, fixing $\varepsilon>0$ small enough, there exists $\overline{\mu}>0$ such that for every $\mu\geq\overline{\mu}$
\begin{equation*}
F_{p,q}(f_{\mu},\G)-E(\phi_{\mu},\R)=E(f_\mu,\R)-E(\phi_{\mu},\R)-\f{1}{q}|f_\mu(0)|^{q}\le \varepsilon-\f{1}{q}|\phi_{\mu}(0)|^{q}<0\,.\qedhere
\end{equation*}
\end{proof}

The rest of the section is organized in two subsections. 

\subsection{Graphs with at least 3 half--lines: proof of Theorem \ref{thm:metric1}} To prove Theorem \ref{thm:metric1} we consider, for $l>0$, graphs $G_l$ as in Figure \ref{fig:metr1}. Let $\vv$ be a vertex of degree 2 from which emanate two bounded edges, say $e_1,e_2$, both of length $l$. Let $\vv_1$ be the vertex of $e_1$ not shared with $e_2$, and let $\vv_2$ be the vertex of $e_2$ not shared with $e_1$. Then, a couple of half--lines emanates both from $\vv_1$ and from $\vv_2$. In particular, we denote by $\HH_{1},\HH_2$ the half-lines emanating from $\vv_{1}$ and by $\HH_{3},\HH_{4}$ those emanating from $\vv_{2}$. Clearly, the graph $\G_l$ satisfies the hypotheses of Theorem \ref{thm:metric1}.

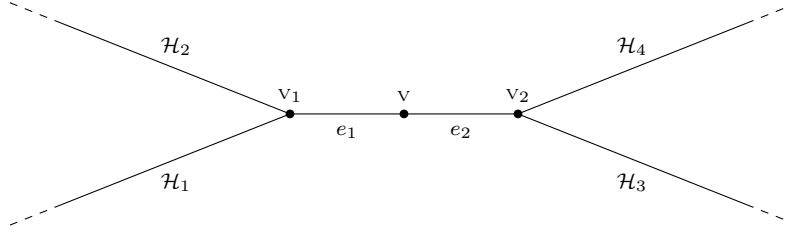
\begin{figure}[t]
	\centering
	
	\begin{tikzpicture}[xscale=0.5,yscale=0.6]
	\node at (0,0) [nodo] {};
	\node at (-3,0) [nodo] {};
	\node at (3,0) [nodo] {};
	\draw (0,0)--(-3,0);
	\draw (0,0)--(3,0);
	\draw (-3,0)--(-9,2);
	\draw (-3,0)--(-9,-2);
	\draw[dashed] (-9,2)--(-10.5,2.5);
	\draw[dashed] (-9,-2)--(-10.5,-2.5);
	\draw (3,0)--(9,2);
	\draw (3,0)--(9,-2);
	\draw[dashed] (9,2)--(10.5,2.5);
	\draw[dashed] (9,-2)--(10.5,-2.5);	
	
	\node at (0,.4) [infinito] {$\scriptstyle\vv$};
	\node at (-3,.4) [infinito] {$\scriptstyle\vv_1$};
	\node at (3,.4) [infinito] {$\scriptstyle\vv_2$};
	\node at (-1.5,-.4) [infinito] {$\scriptstyle e_1$};
	\node at (1.5,-.4) [infinito] {$\scriptstyle e_2$};
	\node at (-6,-1.5) [infinito] {$\scriptstyle \HH_1$};
	\node at (-6,1.5) [infinito] {$\scriptstyle\HH_2$};
	\node at (6,-1.5) [infinito] {$\scriptstyle \HH_3$};
	\node at (6,1.5) [infinito] {$\scriptstyle\HH_4$};
	
	\end{tikzpicture}
	\caption{The graph $\G_l$ in the proof of Theorem \ref{thm:metric1}: a single vertex with degree $2$ attached to two bounded edges of equal length $l$, from the other endpoints of which emanate two couples of half--lines.}
	\label{fig:metr1}
\end{figure}

The following two propositions highlight the dependence of the ground states problem \eqref{eq:problem} for $\G_l$ on the length $l$ of the bounded edges emanating from the vertex of degree $2$.
\begin{proposition}
	\label{prop:metrex}
	Let $\G_l$ be the graph in Figure \ref{fig:metr1}, with $e_1,e_2$ each of length $l$, and $q<\f p2+1$. Then there exists $\overline{l}>0$ (depending on $p,q$) so that, for every $l\geq\overline{l}$, ground states of $F_{p,q}$ at mass $\mu$ on $\G_l$ exist for every $\mu$.
\end{proposition}
\begin{proof}
	The proof is divided in three steps.
	
	{\em Step 1.} We first show that there exists $\overline{\mu}>0$ independent of $l$ so that, for every $l\geq1$, ground states of $F_{p,q}$ at mass $\mu$ on $\G_l$ exist for every $\mu\geq\overline{\mu}$. By Lemma \ref{exRcompactsupp}, there is $\overline{\mu}>0$ such that for every $\mu\ge\overline{\mu}$ there exists a function $f_{\mu}\in H^{1}_{\mu}(\R)$ compactly supported on $[-1,1]$ and satisfying $F_{p,q}(f_{\mu}, \R)<E(\phi_{\mu},\R)$. Since, for every $l\geq1$, one can think of $f_\mu$ as a function in $H_\mu^1(\G_l)$ supported on $e_1\cup e_2$ and centered at $\vv$, by Corollary \ref{compactcor} it follows that a ground state of $F_{p,q}$ at mass $\mu$ on $\G_l$ exists for every $\mu\geq\overline{\mu}$ and $l\geq1$.  
	
	\textit{Step 2.} We now prove that there exists $\underline{\mu}>0$ independent of $l$ so that, for every $l\geq1$, ground states of $F_{p,q}$ at mass $\mu$ on $\G_l$ exist for every $\mu\leq\underline{\mu}$. To this end, let $\widetilde{\G}_{l}$ denote the graph obtained by removing the vertex $\vv$ from $\G_{l}$ and replacing the edges $e_1$, $e_2$ with a single edge $e$ of length $2l$. Hence, $\widetilde{\G}_l$ has two vertices, $\vv_1$ and $\vv_2$. Clearly, $H^{1}(\G_{l})=H^{1}(\widetilde{\G}_{l})$ and for every $u\in H^{1}(\G_{l})$
	\begin{equation}
	\label{FGl<FGtildel}
	F_{p,q}(u,\G_{l})\le F_{p,q}(u,\widetilde{\G}_{l})\,.
	\end{equation}
	Fix now $l=1$ and $\mu>0$, and set $u_{1}\in H^{1}_{\mu}(\widetilde{\G}_{1})$ to be
	\begin{equation}
	\label{u1}
	u_1(x):=\begin{cases}
	\phi_\nu(x) & \text{ if }x\in\HH_i,\text{ for some }i=1,\dots,\,4\\
	\|\phi_\nu\|_{L^\infty(\R)} & \text{ if }x\in e\,,
	\end{cases}
	\end{equation}
	where $\phi_\nu$ denotes the soliton of mass $\nu$ on $\R$ as in \eqref{eq:phimu}. Since $\|u_1\|_{L^2(\widetilde{\G}_1)}^2=\mu$, then $\mu=2\nu+2|\phi_1(0)|^2\nu^{2\alpha}$. In particular, if $\mu\to0$ then $\nu^{2\alpha}=o(\nu)$ and $2\nu=\mu+o(\mu)$. Moreover, since $\alpha q<2\beta+1<\alpha p$ (recall Remark \ref{rem:diagpq}),
	\begin{equation*}
	\begin{split}
	F_{p,q}(u_{1},\widetilde{\G}_{1})-E(\phi_{\mu},\R)&=2E(\phi_{\nu},\R)-\f{2|\phi_1(0)|^p}{p}\nu^{\alpha p}-\f{2|\phi_1(0)|^q}{q}\nu^{\alpha q}-E(\phi_{\mu},\R)\\
	&=-\theta_{p}\left(2^{2\beta}-1\right)\mu^{2\beta+1}+o\left(\mu^{2\beta+1}\right)-\f{2|\phi_1(0)|^q}{q}\left(\f \mu2\right)^{\alpha q}+o\left(\mu^{\alpha q}\right)\\
	&<-\f{|\phi_1(0)|^q}{q}\left(\f\mu2\right)^{\alpha q}+o\left(\mu^{\alpha q}\right)<0\quad\text{as}\quad \mu\to0,
	\end{split}
	\end{equation*}
	which, by Corollary \ref{compactcor} and \eqref{FGl<FGtildel}, entails the existence of $\underline{\mu}>0$ such that ground states of $F_{p,q}$ at mass $\mu$ on $\G_1$ exist for every $\mu\leq\underline{\mu}$.
	
	\begin{figure}[t]
		\centering
		\subfloat{\includegraphics[width=0.41\textwidth]{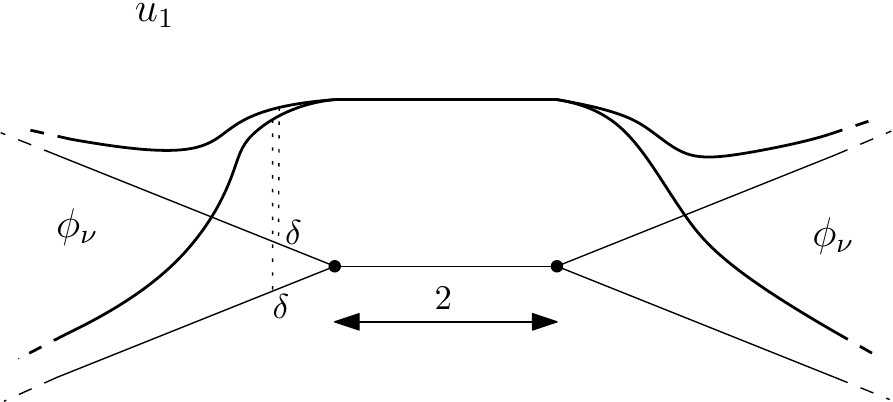}}
		\quad\quad\,\,
		\subfloat{\includegraphics[width=0.48\textwidth]{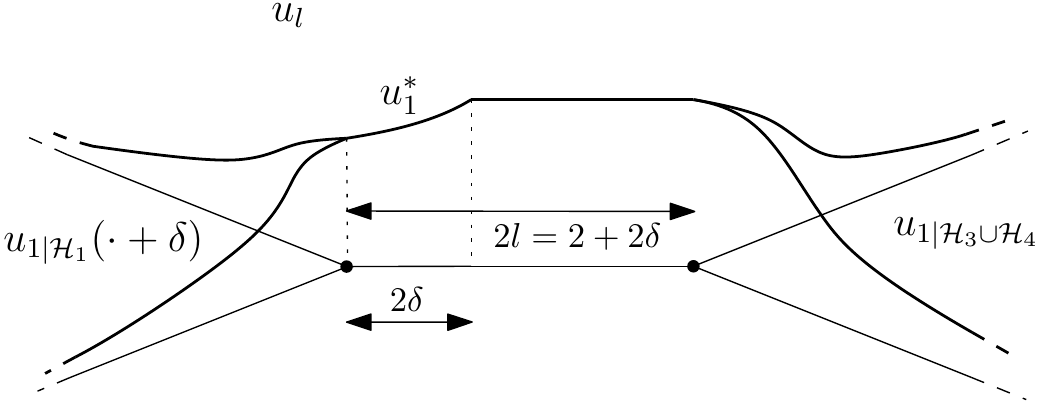}}
		\caption{The functions $u_1$ and $u_l$ as in the proof of Proposition \ref{prop:metrex}.}
		\label{fig:ul}
	\end{figure}
	
	Let now $l>1$. Denote by $\delta:=l-1$ and $J:=\left(\HH_1\cap(0,\delta)\right)\cup\left(\HH_2\cap(0,\delta)\right)$ be the union of the first portion of length $\delta$ of the half--lines $\HH_1$ and $\HH_2$. Let then $u_1$ be the function on $\widetilde{\G}_1$ defined in \eqref{u1} above. Since $u_1$ is symmetrically decreasing on $\HH_{1}\cup\HH_{2}$, the decreasing rearrangement $u_1^*\in H^1(0,2\delta)$ of the restriction of $u_1$ to $J$ satisfies
	\[
	\begin{split}
	\|u_{1}'\|_{L^2(J)}\geq&\,\|(u_{1}^{*})'\|_{L^2\left(0,2\delta\right)}\,,\quad\|u_{1}\|_{L^r(J)}=\|u_{1}^{*}\|_{L^r\left(0,2\delta\right)}\quad\forall r\geq1\\
	u_1^*(0)=&\,\|\phi_\nu\|_{L^\infty(\R)}\,,\qquad\qquad u_1^*(2\delta)=u_{1|\HH_{i}}(\delta),\quad i=1,2\,.
	\end{split}
	\]
	We then parameterize the edge $e$ with $[0,2l]$ so that $\vv_2$ corresponds to $0$ and $\vv_1$ to $2l$ and define $u_{l}\in H^{1}_{\mu}(\widetilde{\G}_{l})$ as (see Figure \ref{fig:ul})
	\begin{equation*}
	u_{l}(x):=
	\begin{cases}
	u_{1}(x), & \text{ if }x\in \HH_{3}\cup\HH_{4}\\
	u_{1}(x), & \text{ if }x\in (0,2]\cap e\\
	u_{1}^{*}(x-2), &\text{ if } x\in (2,2l)\cap {e}\\
	u_{1|\HH_{i}}\left(x+\delta\right), &\text{ if } x\in \HH_{i},\, i=1,2\,.
	\end{cases}
	\end{equation*} 
	Observe that by construction $u_l(\vv_2)=u_1(\vv_2)$ and
	\begin{equation*}
	F_{p,q}(u_{l}, \widetilde{\G}_{l})-E(\phi_{\mu},\R)< E(u_1,\widetilde{\G}_1)-\f{|u_1(\vv_2)|^q}q- E(\phi_\mu,\R)\leq -\f{|\phi_1(0)|^q}{q}\left(\f\mu2\right)^{\alpha q}+o\left(\mu^{\alpha q}\right)<0
	\end{equation*}
	for every $\mu\leq\underline{\mu}$. 
	By \eqref{FGl<FGtildel} and Corollary \ref{compactcor} it then follows that ground states of $F_{p,q}$ at mass $\mu$ on $\G_l$ exist, for every $l\ge 1$ and every $\mu\le \underline{\mu}$.
	
	{\em Step 3. } We now show that there exists $\overline{l}\geq1$ so that, for every $l\geq\overline{l}$, ground states of $F_{p,q}$ on $\G_l$ exist for every $\underline{\mu}\leq\mu\leq\overline{\mu}$, where $\overline{\mu}$, $\underline{\mu}$ are as in Step 1 and 2 above. In view of the previous discussion, this will conclude the proof of Proposition \ref{prop:metrex}.
	
	Let $\phi_{\underline{\mu}}$ be the soliton at mass $\underline{\mu}$ on the real line. For every $l$, let $\delta=\delta(l)$ and $\kappa=\kappa(l)$ be such that $w_{l}(x):=\kappa(\phi_{\underline{\mu}}(x)-\delta)_{+}$ is compactly supported on $(-l,l)$ and $\|w_l\|_{L^2(-l,l)}^{2}=\underline{\mu}$. In particular, observe that $\delta\to 0$, $\kappa\to 1$ and $w_{l}-\phi_{\underline{\mu}}\to 0$ strongly in $H^1(\R)$ as $l \to +\infty$. Hence, thinking of $w_l$ as a function in $H_\mu^1(\G_l)$ supported on $e_1\cup e_2$, we have
	\begin{equation}
	\label{Fpqwlgl}
	F_{p,q}(w_{l},\G_{l})-E(\phi_{\underline{\mu}},\R)=E(w_{l},\R)-E(\phi_{\underline{\mu}},\R)-\f{1}{q}|w_{l}(0)|^{q}=-\f{|\phi_{\underline{\mu}}(0)|^q}q+o(1)\\
	\end{equation}
	as $l\to+\infty$, implying by Corollary \ref{compactcor} the existence of $\overline{l}>0$ so that ground states of $F_{p,q}$ on $\G_l$ at mass $\underline{\mu}$ exist for every $l\geq\overline{l}$. 
	
	Let now $\underline{\mu}<\mu\leq\overline{\mu}$ be fixed. Since, for every $l\geq\overline{l}$, the function $w_l\in H_{\underline{\mu}}^1(\G_l)$ above is supported on $e_1\cup e_2$ only, setting $w_{l,\mu}$ to be
	\begin{equation*}
	w_{l,\mu}(x):=\left(\f{\mu}{\underline{\mu}}\right)^{\alpha}\widetilde{w}_{l}\left(\left(\f{\mu}{\underline{\mu}}\right)^{\beta}x\right),
	\end{equation*}
	we can think of $w_{l,\mu}$ as a function in $H_\mu^1(\G_l)$ supported on $e_1\cup e_2$ (and identically equal to zero both on a suitable portion of $e_1$ close to $\vv_1$ and on the corresponding final portion of $e_2$ close to $\vv_2$). By \eqref{Fpqwlgl}, Remark \ref{rem:omot} and the convergence of $w_l$ to $\phi_{\underline{\mu}}$ as $l\to+\infty$, we have that
	\begin{equation*}
	\begin{split}
	F_{p,q}(w_{l,\mu},\G_{l})-E(\phi_{\mu},\R)&=\left(\f{\mu}{\underline{\mu}}\right)^{2\beta+1}\left(E(w_{l},\R)-E(\phi_{\underline{\mu}},\R)\right)-\f{1}{q}\left(\f{\mu}{\underline{\mu}}\right)^{\alpha q}|w_l(0)|^q\\
	&=-\f{1}{q}\left(\f{\mu}{\underline{\mu}}\right)^{\alpha q}|\phi_{\underline{\mu}}(0)|^{q}+ o(1)<0\\
	\end{split}
	\end{equation*}
	for every $\mu\in(\underline{\mu},\overline{\mu}]$, provided $l$ is sufficiently large. Hence, up to possibly enlarging the value of $\overline{l}$, by Corollary \ref{compactcor} we conclude that ground states of $F_{p,q}$ at mass $\mu$ on $\G_l$ exist for every $\underline{\mu}\leq\mu\leq\overline{\mu}$ and for every $l\geq\overline{l}$. This gives the claim of Step 3 and completes the proof.
\end{proof}

\begin{proposition}
	\label{prop:metrnonex1}
	Let $\G_l$ be the graph in Figure \ref{fig:metr1}, with $e_1,e_2$ each of length $l$, and $q<\f p2+1$. Then there exists $\underline{l}>0$ (depending on $p,q$) and a value $m>0$ so that for every $l\leq\underline{l}$ ground states of $F_{p,q}$ at mass $m$ on $\G_l$ do not exist.
\end{proposition}
\begin{proof}
	Choose $m$ so that $m>\mu^{*}$, where $\mu^{*}=\mu^{*}(p,q,3,4)$ is the critical value associated to $F_{p,q,3}$ on the star graph $S_4$ with $4$ half--lines as in Proposition \ref{exstarg}. Hence, by Proposition \ref{exstarg} and Corollary \ref{compactcor}, ground states of $F_{p,q,3}$ at mass $m$ on $S_4$ do not exist and 
	\begin{equation}
	\label{Fpqw>sol}
	F_{p,q,3}(w, S_{4})>-\theta_{p} m^{2\beta+1}\,,\qquad \forall\, w\in H^{1}_{m}(S_{4})\,.
	\end{equation} 
	Let us now show that ground states of $F_{p,q}$ at mass $m$ on $\G_l$ do not exist, provided $l$ is small enough.
	
	Assume by contradiction that this is not the case and that there exists a ground state $u_{l}$ of $F_{p,q}$ at mass $m$ on $\G_l$ for every $l>0$. By Corollary \ref{compactcor}, this entails that for every $l>0$
	\begin{equation}
	\label{Fpqul-bdd}
	F_{p,q}(u_{l}, \G_{l})\le E(\phi_{m}, \R)=-\theta_{p}m^{2\beta+1}\,,
	\end{equation}
	which coupled with \eqref{GnpG}--\eqref{GninfG} and the fact that $p\in(2,6)$, $q\in(2,4)$, ensures that $\|u_{l}\|_{H^{1}(\G_{l})}$ is bounded from above uniformly on $l$. 
	
	Let then $\K_l:=e_1\cup e_2$. On the one hand, by \eqref{GninfG} we get
	\begin{equation}
	\label{Eul-K}
	\liminf_{l\to 0} E(u_{l}, \K_{l})\ge -\liminf_{l\to 0}\f{2l\|u_{l}\|_{L^{\infty}(\K_{l})}^{p}}{p}\ge -\liminf_{l\to 0}\f{2l m^{\f{p}{4}}\|u'_{l}\|_{L^{2}(\G_{l})}^{\f{p}{2}}}{p}=0\,.
	\end{equation}
	On the other hand, for every $l$, since $u_{l}$ is a ground state of $F_{p,q}$, then it is non--increasing along every half--line of $\G_{l}$ by Lemma \ref{dec-2half}. Hence, by \cite[Proposition 1.7.1]{cazenave} there exist $u_1,\,u_2\in H^{1}(\R)$ such that, up to subsequences, $u_{l|\HH_{1}\cup\HH_{2}}\to u_1$ and $u_{l|\HH_{3}\cup\HH_{4}}\to u_2$ strongly in $L^{r}(\R)$ for every $2<r\le \infty$. Moreover, for every $x,y\in\K_l$, by H\"older inequality
	\begin{equation}
	\label{eq:holder}
	|u_{l}(x)-u_{l}(y)|\le \sqrt{2l}\|u'_{l}\|_{L^{2}(\G_{l})}\to 0\quad\text{as}\quad l\to 0\,.
	\end{equation} 
	In particular, this implies that $u_1(0)=\lim_{l\to0}u_{l|\HH_{1}\cup\HH_{2}}(\vv_1)=\lim_{l\to0}u_{l|\HH_{3}\cup\HH_{4}}(\vv_2)=u_2(0)$. Therefore, letting as usual $h_1,h_2,h_3,h_4$ be the half--lines of the star graph $S_4$, we define $u\in H^1(S_4)$ as
	\[
	u(x):=\begin{cases}
	u_1(x) & \text{ if }x\in h_1\cup h_2\\
	u_2(x) & \text{ if }x\in h_3\cup h_4\,.
	\end{cases}
	\]
	By construction, it then follows that $\|u\|_{L^2(S_4)}\leq m$ and by \eqref{Eul-K}
	\begin{equation}
	\label{Eul-G-K}
	\liminf_{l\to 0}E(u_{l}, \G_{l})\ge E(u, S_{4}).
	\end{equation}
	Furthermore, \eqref{eq:holder} implies that $u_{l}(\ww)\to u(0)$ for every vertex $\ww\in \G_{l}$. Since $\G_l$ has $3$ vertices, by \eqref{Fpqul-bdd} and \eqref{Eul-G-K} we get
	\begin{equation}
	\label{EuS4}
	\begin{split}
	F_{p,q,3}(u, S_{4})=E(u,S_{4})-\f{3}{q}|u(0)|^{q}\le&\, \liminf_{l\to 0}\left(E(u_l,\G_l)-\f{|u_l(\vv_1)|^q+|u_l(\vv)|^q+|u_l(\vv_2)|^q}{q}\right)\\
	=&\,\liminf_{l\to 0}F_{p,q}(u_{l},\G_{l})\le -\theta_{p}m^{2\beta+1}.
	\end{split}
	\end{equation}
	Note that \eqref{EuS4} immediately implies $u\not\equiv0$ on $S_4$. Also, if it were $\|u\|_{L^{2}(S_{4})}^{2}<m$, then there would exist $\beta>1$ so that $\beta u\in H^{1}_{m}(S_{4})$ and
	\begin{equation*}
	-\theta_{p}m^{2\beta+1}< F_{p,q,3}(\beta u, S_{4})<\beta^{2}F_{p,q,3}(u, S_{4})<F_{p,q,3}(u, S_{4})\le -\theta_{p}m^{2\beta+1}
	\end{equation*}
	by \eqref{Fpqw>sol} and \eqref{EuS4} (recalling also that $p\in(2,6)$, $q\in (2,4)$), which is impossible. Hence, it must be $\|u\|_{L^{2}(S_{4})}^{2}=m$. But  this means that $u\in H^{1}_{m}(S_{4})$ satisfies \eqref{EuS4}, contradicting \eqref{Fpqw>sol}. Therefore, there exists $\underline{l}>0$ such that no ground states of $F_{p,q}$ at mass $m$ exist on $\G_l$, for every $l\leq\underline{l}$, and we conclude.
\end{proof}
\begin{proof}[Proof of Theorem \ref{thm:metric1}]
	To exhibit a graph $\G^1$ as in part (i), by Proposition \ref{prop:metrex} it is enough to take $\G_l$ for sufficiently large value of $l$. Conversely, to prove statement (ii) one can simply take $\G_l$ and $m$ as in Proposition \ref{prop:metrnonex1}.
\end{proof}

\subsection{Graphs with exactly two half--lines: proof of Theorems \ref{thm:metric2}--\ref{thm:metric3}}
We begin with the proof of Theorem \ref{thm:metric2}. To this end, given $k\in \N$, we consider the graph $\G_k$ as in Figure \ref{fig:metr2}. The three vertices of $\G_k$ are denoted by $\vv_1$, $\vv_2$, $\vv_3$. The vertices $\vv_1$ and $\vv_2$ share $3$ edges of length $1$. The vertices $\vv_2$ and $\vv_3$ share $k$ edges of length 1. The compact core of $\G_k$ will be denoted by $\K_k$ The two half--lines $\HH_1$, $\HH_2$ of $\G_k$ emanate from the vertex $\vv_3$. Clearly, $\G_k$ fulfills the hypotheses of Theorem \ref{thm:metric2} for every $k\geq2$.

\begin{figure}[t]
	\centering
	\includegraphics[height=0.3\textheight]{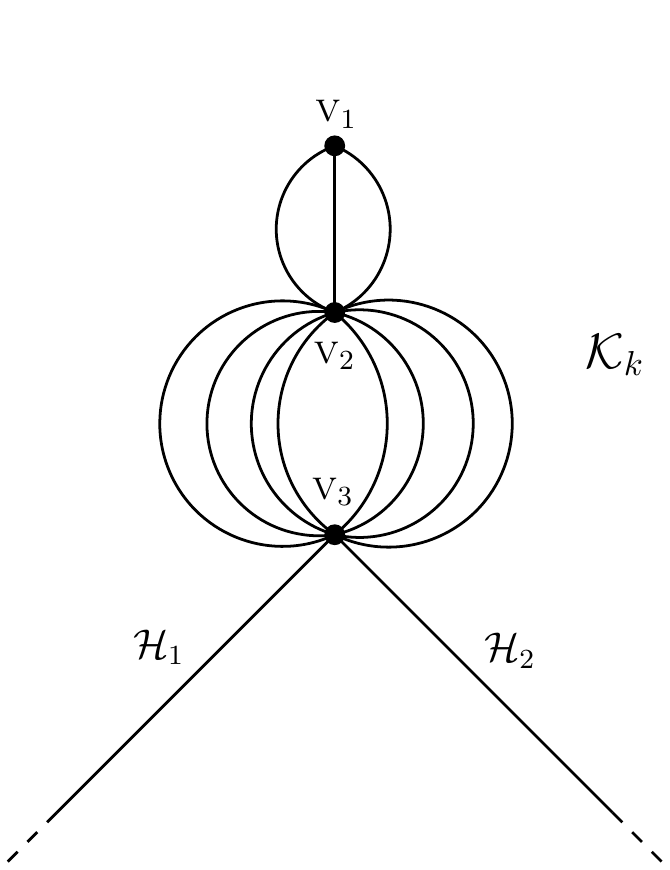}
	\caption{The graph $\G_k$ as in the proof of Theorem \ref{thm:metric2}. For every $k\in\N$, the compact core $\K_k$ of $\G_k$ is given by $3$ edges of length $1$, attached to the vertices $\vv_1$ and $\vv_2$, and $k$ edges of length $1$, attached to the vertices $\vv_2$ and $\vv_3$. The two half--lines of $\G_k$ emanate from the same vertex $\vv_3$.}
	\label{fig:metr2}
\end{figure}

\begin{proof}[Proof of Theorem \ref{thm:metric2}]
	We claim that there exists $m>0$ so that ground states of $F_{p,q}$ at mass $m$ on $\G_k$ do not exist, provided $k$ is sufficiently large. If this is true, then to exhibit a graph $\G$ and a mass $m$ as in the statement of Theorem \ref{thm:metric2} it is enough to take $\G=\G_k$ for a suitably large value of $k$.
	
	Let us thus prove the claim, dividing the argument in the following steps.
	
	{\em Step 1. }Let $m>0$ be fixed. The actual value of $m$ will be properly chosen later on in the argument. Assume by contradiction that a ground state $u_k\in H_m^1(\G_k)$ at mass $m$ exists on $\G_k$ for every $k\in\N$, so that by Corollary \ref{compactcor}
	\begin{equation}
		\label{unGS}
		F_{p,q}(u_k,\G_k)\leq-\theta_p m^{2\beta+1}\qquad\forall k\in\N\,.
	\end{equation}
	Coupling \eqref{unGS} with \eqref{GnpG}, \eqref{GninfG}, $p\in(2,6)$, $q\in(2,4)$ and $\|u_k\|_{L^2(\G_k)}^2=m$, it then follows that $\|u_k'\|_{L^2(\G_k)}$ is bounded from above uniformly on $k$. In particular, denoting by $\widetilde{\K}_k\subset\K_k$ the union of the $k$ edges of length 1 between $\vv_2$ and $\vv_3$, we have that there exists a suitable constant $C>0$ so that $\|u_k\|_{L^2(\widetilde{\K}_k)}\leq C$ and $\|u_k'\|_{L^2(\widetilde{\K}_k)}\leq C$ independently of $k$. 
	
	For every $k$, consider now an ordering $e_1,\,e_2,\,\dots,\,e_k$ of the edges of $\widetilde{\K}_k$ so that if $i<j$, that is $e_i$ precedes $e_j$ in the ordering, then $\|u_k\|_{L^2(e_i)}\geq\|u_k\|_{L^2(e_j)}$. Then $\|u_k\|_{L^2(e_k)}\to0$ and $\min_{x\in e_{k}} u_k(x)\to0$ as $k\to+\infty$. Hence, by H\"older inequality, for every $x\in e_{k}$
	\[
	u_k^2(x)\leq\left(\min_{y\in e_{k}}u_k(y)\right)^2+2\|u_k\|_{L^2(e_{k})}\|u_k'\|_{L^2(e_{k})}\to0\qquad\text{ as }k\to+\infty\,.
	\]
	In particular,
	\begin{equation}
		\label{v2v3to0}
		u_k(\vv_2)\to0\quad\text{ and }\quad u_k(\vv_3)\to0\quad\text{ as }k\to+\infty\,.
	\end{equation}
	
	{\em Step 2. } Since $u_k$ is a ground state on $\G_k$, by Lemma \ref{dec-2half} it is non--increasing on both $\HH_{1}$ and $\HH_{2}$. Hence, \eqref{v2v3to0} and $\|u_k\|_{L^2(\G_k)}^2=m$ yields $\|u_k\|_{L^p(\HH_{1}\cup\HH_{2})}\to0$ as $k\to+\infty$, so that
	\begin{equation}
		\label{En_H12}
		\liminf_{k\to+\infty}E(u_k,\HH_{1}\cup\HH_{2})\geq0\,.
	\end{equation}
	Furthermore, since $\|u_k\|_{H^1(\G_k)}$ is bounded uniformly on $k$ and, for every $k$, $\K_k\setminus\widetilde{\K}_k$ is given by $3$ edges of length 1, there exists $w_1\in H^1(\K_k\setminus\widetilde{\K}_k)$ so that $u_{k|\K_k\setminus\widetilde{\K}_k}\rightharpoonup w_1$ in $H^1(\K_k\setminus\widetilde{\K}_k)$ and  $u_{k|\K_k\setminus\widetilde{\K}_k}\to w_1$ strongly in $L^r$ for every $r\geq2$ as $k\to+\infty$. By \eqref{v2v3to0}, $w_1(\vv_2)=0$. Observe that we can think of $w_1$ as a function on the star graph $S_3$ with 3 half--lines supported on the ball of radius 1 centered at the unique vertex of $S_3$. Therefore we have
	\begin{equation}
	\label{EunS3}
	\liminf_{k\to+\infty}\left(E(u_k,\K_k\setminus\widetilde{\K}_k)-\f{|u_k(\vv_1)|^q}q\right)\geq F_{p,q}(w_1,S_3)\,.
	\end{equation}
	Let now	
	\[
	\begin{split}
	\widetilde{\K}_{k,1}&\,:=\left\{e_j\in\widetilde{\K}_k\,:\,\lim_{k\to+\infty}\|u_k\|_{L^2(e_j)}\neq0\right\}\\
	\widetilde{\K}_{k,2}&\,:=\left\{e_j\in\widetilde{\K}_k\,:\,\lim_{k\to+\infty}\|u_k\|_{L^2(e_j)}=0\right\}\,,
	\end{split}
	\]
	so that $\widetilde{\K}_k=\widetilde{\K}_{k,1}\cup\widetilde{\K}_{k,2}$ for every $k$. Note that either $\widetilde{\K}_{k,1}=\widetilde{\K}_k$ and $\widetilde{\K}_{k,2}=\emptyset$ for every $k$, or $\widetilde{\K}_{k,1}$ contains a number of edges which is bounded from above uniformly on $k$.
	
	Arguing as in the final part of Step 1, we obtain $\|u_k\|_{L^\infty(\widetilde{\K}_{k,2})}\to0$ as $k\to+\infty$, which implies
	\[
	\|u_k\|_{L^p(\widetilde{\K}_{k,2})}^p\leq\|u_k\|_{L^2(\widetilde{\K}_{k,2})}^2\|u_k\|_{L^\infty(\widetilde{\K}_{k,2})}^{p-2}\leq m \|u_k\|_{L^\infty(\widetilde{\K}_{k,2})}^{p-2}\to0 \qquad\text{ as }n\to+\infty\,,
	\]
	thus yielding
	\begin{equation}
		\label{EunKn2}
		\liminf_{k\to+\infty}E(u_k,\widetilde{\K}_{k,2})\geq0\,.
	\end{equation}
	On the contrary, note that for every $e_j\in\widetilde{\K}_{k,1}$, there exists $w_{e_j}\in H^1(0,1)$, $w_{e_j}\not\equiv0$ on $(0,1)$, so that $\lim_{x\to0^+}w_{e_j}(x)=\lim_{x\to1^-}w_{e_j}(x)=0$, $u_{k|e_j}\rightharpoonup w_{e_j}$ in $H^1(0,1)$ and $u_{k|e_j}\to w_{e_j}$ strongly in $L^r(0,1)$ for every $r\geq2$ as $k\to+\infty$. Let then $l:=\lim_{n\to+\infty}|\widetilde{\K}_{n,1}|$, where as usual $|\widetilde{\K}_{k,1}|$ denotes the length of $\widetilde{\K}_{k,1}$ (which is also the number of edges in $\widetilde{\K}_{k,1}$ since each edge is of length 1), and note that either $l\in\N$ or $l=+\infty$. Writing then $\widetilde{\K}_{k,1}=\bigcup_{j=1}^l e_j$, we consider $w_2\in H^1(\R)$ given by
	\[
	w_2(x):=\begin{cases}
	w_{e_j}(x) & \text{if }x\in[j-1,j],\text{ for some }j\in\N,\,1\leq j\leq l\\
	0 & \text{otherwise}\,.
	\end{cases}
	\]
	By construction, 
	\[
	\liminf_{k\to+\infty}\|u_k'\|_{L^2(\widetilde{\K}_{k,1})}\geq\|w_2'\|_{L^2(\R)}\quad\text{ and }\quad\lim_{k\to+\infty}\|u_k\|_{L^r(\widetilde{\K}_{k,1})}=\|w_2\|_{L^r(\R)}\quad\forall r\geq2\,,
	\]
	so that
	\begin{equation}
		\label{EunKn1}
		\liminf_{k\to+\infty}E(u_n,\widetilde{\K}_{k,1})\geq E(w_2,\R)\,.
	\end{equation}
	
	{\em Step 3. }Let $w_1$, $w_2$ be the functions defined in Step 2. Clearly $\|w_1\|_{L^2(S_3)}^2+\|w_2\|_{L^2(\R)}^2\leq m$. Let us now choose $m<\mu^*$, where $\mu^*=\mu^*(p,q,1,3)$ is the critical value associated to $F_{p,q,1}$ (which is indeed $F_{p,q})$ on $S_3$ by Proposition \ref{exstarg}. Hence, since $q>\f p2+1$, ground states of $F_{p,q}$ at mass $\mu$ do not exist on $S_3$ for every $\mu\leq m$. In particular, ground states of $F_{p,q}$ do not exist on $S_3$ at mass $\|w_1\|_{L^2(S_3)}^2$, so that $F_{p,q}(w_1,S_3)>\ee\left(\|w_1\|_{L^2(S_3)}^2,\R\right)$. Coupling with \eqref{EunS3} leads to
	\begin{equation}
	\label{ass1}
		\liminf_{k\to+\infty}\left(E(u_k,\K_k\setminus\widetilde{\K}_k)-\f{|u_k(\vv_1)|^q}q\right)>\ee\left(\|w_1\|_{L^2(S_3)}^2,\R\right)=-\theta_p\left(\|w_1\|_{L^2(S_3)}^2\right)^{2\beta+1}\,.
	\end{equation}
	Moreover, since $w_2\in H_{\|w_2\|_{L^2(\R)}^2}^1(\R)$, then $E(w_2,\R)\geq\ee(\|w_2\|_{L^2(\R)}^2,\R)$, and by \eqref{EunKn2} and \eqref{EunKn1} we get
	\begin{equation}
		\label{ass2}
		\liminf_{k\to+\infty}E(u_n,\widetilde{\K}_{k})\geq\ee(\|w_2\|_{L^2(\R)}^2,\R)=-\theta_p\left(\|w_2\|_{L^2(\R)}^2\right)^{2\beta+1}\,.
	\end{equation}
	Hence, combining \eqref{ass1}--\eqref{ass2} with \eqref{En_H12} and $\eqref{v2v3to0}$ gives
	\[
	\begin{split}
	\liminf_{k\to+\infty}F_{p,q}(u_k,\G_k)=&\,\liminf_{k\to+\infty}\left(E(u_k,\K_k\setminus\widetilde{\K}_k)-\f{|u_k(\vv_1)|^q}q\right)\\
	&\,+\liminf_{k\to+\infty}E(u_k,\widetilde{\K}_{k,1})+\liminf_{k\to+\infty}E(u_k,\widetilde{\K}_{k,2})\\
	>&\,-\theta_p\left(\|w_1\|_{L^2(S_3)}^2\right)^{2\beta+1}-\theta_p\left(\|w_2\|_{L^2(\R)}^2\right)^{2\beta+1}>-\theta_{p} m^{2\beta+1}\,.
	\end{split}
	\]
	Since this contradicts \eqref{unGS}, we conclude.
\end{proof}
At the end of this section, we prove Theorem \ref{thm:metric3}. To this end, we will consider a graph $\G_l$ as in Figure \ref{fig:metr3}. For every $l>0$, such a graph is obtained by choosing $k\in\N$ large enough so that Theorem \ref{thm:metric2} applies to the graph $\G_k$ as in Figure \ref{fig:metr2} and then adding a vertex with degree $2$, attached to two bounded edges each of length $l$ between the original compact core of $\G_k$ and the two half--lines of the graph.

\begin{figure}[t]
	\centering
	\includegraphics[height=0.25\textheight]{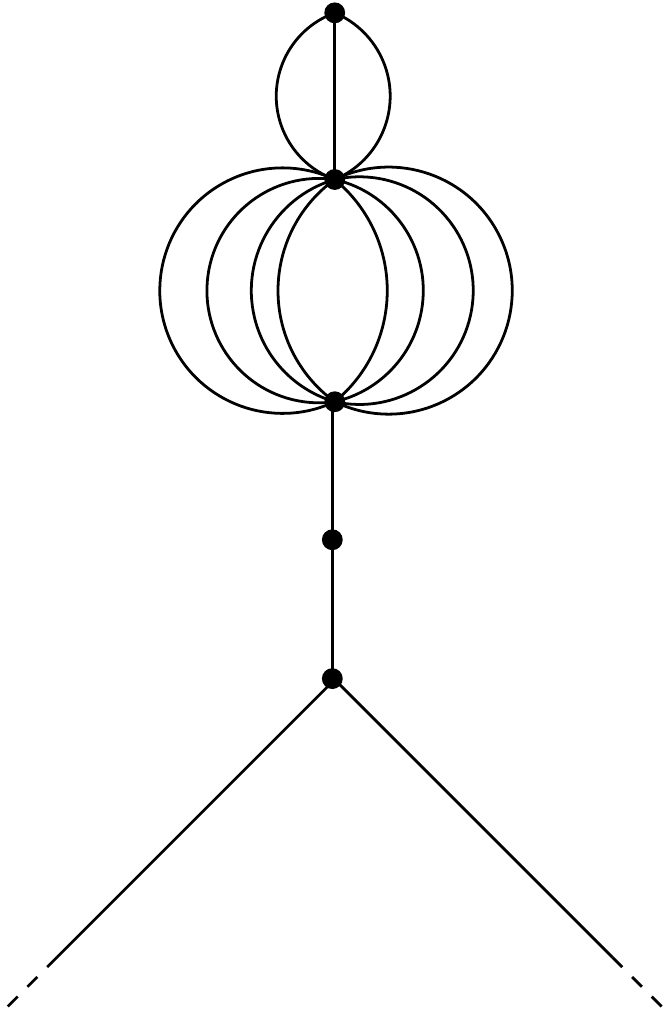}
	\caption{The graph $\G_l$ in the proof of Theorem \ref{thm:metric3}. It is obtained by adding a vertex with degree $2$, attached to two edges each of length $l$, to the graph in Figure \ref{fig:metr2} for which Theorem \ref{thm:metric2} holds.}
	\label{fig:metr3}
\end{figure}

\begin{proof}[Proof of Theorem \ref{thm:metric3}]
	To prove the theorem we can simply show that there exist $m>0$ and $\underline{l}>0$ so that ground states of $F_{p,q}$ at mass $m$ do not exist on $\G_l$, for every $l\leq\underline{l}.$ To this end, it is enough to take $m$ as in Theorem \ref{thm:metric2} and then adapting the argument in the proof of Proposition \ref{prop:metrnonex1}, since as $l\to0$ the limiting graph $\G_0$ admits no ground states at mass $m$ by Theorem \ref{thm:metric2}.
\end{proof}

\appendix
\section{A useful identity}
\label{sec:appA}
The following identity concerning the actual value of the energy of the soliton $\phi_1$ at mass $\mu=1$ for the energy $E$ with the sole standard nonlinearity on the real line is crucial in the proof of Theorem \ref{2half}. Perhaps the result is well--known, but we do not have any explicit reference. Since its proof is elementary, we report it here for the sake of completeness.
\begin{lemma} 
	\label{thetap-2beta}
	For every $p\in(2,6)$, it holds
	\begin{equation*}
	\theta_{p}(2\beta+1)=\f{|\phi_{1}(0)|^{p-2}}{p}.
	\end{equation*}
\end{lemma}
\begin{proof}
	Since $\phi_{1}$ is the ground state of $E(\cdot,\R)$ at mass $\mu=1$, there exists $\omega>0$ such that
	\begin{equation}
	\label{ELEphi1}
	\phi_{1}''+|\phi_{1}|^{p-2}\phi_{1}=\omega \phi_{1}\qquad\text{on }\R.
	\end{equation}
	The conservation of mechanical energy then implies
	\begin{equation*}
	\f{1}{2}(\phi_{1}'(x))^{2}+\f{1}{p}(\phi_{1}(x))^{p}=\f{\omega}{2}(\phi_{1}(x))^{2}
	\end{equation*}
	for every $x\in \R$, so that integrating on $\R$ one gets
	\begin{equation}
	\label{mec-en-phi1}
	\f{1}{2}\|\phi_{1}'\|_{2}^{2}+\f{1}{p}\|\phi_{1}\|_{p}^{p}=\f{\omega}{2}.
	\end{equation}
	Furthermore, multiplying \eqref{ELEphi1} by $\phi_{1}$ and integrating on $\R$ leads to
	\begin{equation}
	\label{nehari-phi1}
	\|\phi_{1}'\|_{2}^{2}-\|\phi_{1}\|_{p}^{p}+\omega=0.
	\end{equation}
	By \eqref{mec-en-phi1} and \eqref{nehari-phi1}, we get 
	\begin{equation}
	\label{eq:thetap}
	\theta_{p}=-E(\phi_{1},\R)=\f{6-p}{2(p+2)}\omega=\f{\omega}{2(2\beta+1)}\,.
	\end{equation}
	Recalling that $\phi_1$ is explicitly given by 
	\begin{equation*}
	\phi_{1}(x)=\left[\f p2\omega\left(1-\tanh^2\left(\f{p-2}2\sqrt{\omega}(|x|)\right)\right)\right]^{\f 1{p-2}},\qquad x\in\R,
	\end{equation*}
	it turns out that $\f{1}{p}|\phi_{1}(0)|^{p-2}=\f{\omega}{2}$ and, coupling with \eqref{eq:thetap}, we conclude.
\end{proof}

\section{Some rearrangement results}
\label{sec:appB}
The next two lemmas collect some general constructions based on the theory of rearrangements on graphs that are helpful when proving the non--existence results in Theorem \ref{3half-deg3}. Recall that $V^+$ is the set of vertices attached to at least one half--line.
\begin{lemma}
	\label{rearr1}
	Let $\G$ be a non--compact graph satisfying Assumption (H) with $N\geq3$ half--lines and $u\in H^{1}_{\mu}(\G)$ be a positive function. Then there exists $u^{*}\in H^{1}_{\mu}(S_{3})$ on the star--graph $S_3$ with $3$ half--lines such that
	\begin{equation}
	\label{ESleEG}
	E(u^{*}, S_{3})\le E(u,\G)
	\end{equation}
	and
	\begin{equation}
	\label{u0leminu}
	u^{*}(0)=\min_{\vv\in V^{+}}u(\vv).
	\end{equation}
	Moreover, $u^*$ is symmetric with respect to the origin and monotonically decreasing on $2$ half--lines of $S_3$, whereas on the remaining half--line it is non--decreasing from the origin to a unique maximum point and then non--increasing from this point on the rest of the half--line.
\end{lemma}
\begin{proof}
	Let $u\in H^{1}_{\mu}(\G)$ be a positive function, $m:=\min_{\vv\in V^{+}} u(\vv)$, $J:=\{x\in \G\,:\, u(x)>m\}$ and $u_{|J}$ be the restriction of $u$ to $J$. Note that $u(J)=(m,\|u\|_{\infty}]$ is connected and, by Assumption (H), every value in $u(J)$ is attained at least twice on $\G$, except possibly $\|u\|_{\infty}$. Denoting by $\widehat{u}\in H^{1}(-L,L)$ the symmetric rearrangement of $u_{|J}$ on the interval $(-L, L)$, with $L:=\f{|J|}{2}$, we have (see \cite[Proposition 3.1]{AST}) 
	\[
	\|u'\|_{L^2(J)}\geq\|\widehat{u}'\|_{L^2(-L,L)}\,,\quad\|u\|_{L^r(J)}=\|\widehat{u}\|_{L^r(-L,L)}\quad\forall r\geq1, \quad \widehat{u}(-L)=\widehat{u}(L)=m\,.
	\]
	Similarly, $u(\G\setminus J)\subseteq[0,m]$ is connected and every value in $u(\G\setminus J)$ is attained at least $N\geq3$ times (i.e. at least once on each half--line). Therefore, letting $\widetilde{u}\in H^1(S_3)$ be the symmetric rearrangement on $S_3$ of $u_{\mid \G\setminus J}$ as in \cite[Appendix A]{acfn_jde}, we get
	\[
	\|u'\|_{L^2(\G\setminus J)}\geq\|\widetilde{u}'\|_{L^2(S_3)}\,,\quad\|u\|_{L^r(\G\setminus J)}=\|\widetilde{u}\|_{L^r(S_3)}\quad \forall r\geq1,\quad \widetilde{u}(0)=m\,.
	\]
	Denoting by $h_{1},h_2,h_3$ the half--lines of $S_3$, set $u^*:S_3\to\R$
	\[
	u^*(x):=\begin{cases}
	\widehat{u}(x-L) & x\in [0,2L]\cap h_1\\
	\widetilde{u}(x-2L) & x\in [2L,+\infty)\cap h_{1}\\
	\widetilde{u}(x) & \text{otherwise}\,.
	\end{cases}
	\]
	By construction, $u^{*}\in H^{1}_{\mu}(S_{3})$, it satisfies \eqref{ESleEG} and \eqref{u0leminu} and enjoys the desired monotonicity and symmetry properties.
\end{proof}
\begin{lemma}
	\label{rearr-halflines}
	Let $\G$ be a non--compact graph with at least $N\ge3$ half--lines and $u\in H^{1}(\G)$ be a positive function. If $\|u\|_{L^{\infty}(\K)}<\|u\|_{L^{\infty}(\G\setminus\K)}$, then there exist two positive functions $u_{1},u_2\in H^{1}(\R^{+})$ such that 
	\begin{equation*}
	2\|u_{1}\|_{L^{2}(\R^{+})}^{2}+\|u_{2}\|_{L^{2}(\R^{+})}^{2}=\|u\|_{L^{2}(\G\setminus\K)}^{2},
	\end{equation*}
	\begin{equation*}
	2E(u_{1},\R^{+})+E(u_{2},\R^{+}) \le E(u,\G\setminus\K)
	\end{equation*}
	and
	\begin{equation*}
	u_{1}(0)=\min_{\vv\in V^{+}} u(\vv)\,,\quad u_{2}(0)=\|u\|_{L^{\infty}(\K)}.
	\end{equation*}
	In particular, $u_{1}$ is decreasing on $\R^+$, while $u_{2}$ is increasing from the origin to a unique point of maximum and then decreasing from this point on the rest of the half--line. 
	
\end{lemma}
\begin{proof}
	Set $m:=\min_{\vv\in V^{+}} u(\vv)$ and $M:=\|u\|_{L^{\infty}(\K)}$. Denoting by $J_{1}:=\{x\in\G\setminus \K\,:\,u(x)>M\}$ and by $u_{|J_{1}}$ the restriction of $u$ to $J_{1}$, then $u(J_{1})$ is connected and, since the $L^\infty$ norm of $\G$ is not attained on the compact core $\K$, every value $t\in u(J_{1})$ is attained at least twice (on any half--line where $u$ attains its $L^\infty$ norm), except possibly $\|u\|_{L^{\infty}(\G\setminus \K)}$. Letting $\overline{u}\in H^1(-L_{1},L_{1})$ be the symmetric rearrangement of $u_{|J_{1}}$ on the interval $(-L_{1},L_{1})$, with $L_1:=\f{|J_{1}|}{2}$, by \cite[Proposition 3.1]{AST} it follows 
	\begin{equation*}
	\|\overline{u}\|_{L^2(-L_1,L_1)}=\|u\|_{L^2(J_1)}\,,\quad E(\overline{u}, (-L_{1},L_{1}))\le E(u_{|J_{1}}, J_{1})\,,\quad \lim_{x\to\pm L_{1}^\mp} u(x)=M.
	\end{equation*}
	Moreover, denoting by $J_{2}:=\{x\in\G\setminus \K\,:\, m< u(x)\le M\}$, we observe that $u(J_{2})\subseteq[m,M]$ is connected and every value $t\in u(J_{2})$ is attained at least once (for instance on any half--line where $u$ attains its $L^\infty$ norm). We thus consider the decreasing rearrangement $\widetilde{u}\in H^1(0,L_2)$ of $u_{|J_{2}}$ on the interval $[0,L_{2})$, where $L_{2}:=|J_{2}|$, so that 
	\begin{equation*}
	\|\widetilde{u}\|_{L^2(0,L_2)}=\|u\|_{L^2(J_2)}\,\quad E(\widetilde{u}, (0,L_{2}))\le E(u_{|J_{2}}, J_{2})\,,\quad \widetilde{u}(0)=M\,,\quad \lim_{x\to L_{2}^-}\widetilde{u}(x)=m.
	\end{equation*}
	Similarly, $u(\G\setminus (\K\cup J_{1} \cup J_{2}))\subset [0,m]$ is connected and every value $t\in u(\G\setminus (\K\cup J_{1} \cup J_{2}))$ is attained at least $N\geq3$ times (i.e. once on each half--line). Therefore, rearranging symmetrically decreasing on $S_3$ the restriction of $u$ to $\G\setminus (\K\cup J_{1} \cup J_{2})$ (see \cite[Appendix A]{acfn_jde}), there exists $u_{1}\in H^{1}(\R^{+})$ satisfying 
	\begin{equation*}
	\begin{split}
	3\|u_1\|_{L^2(\R^+)}=&\,\|u\|_{L^2(\G\setminus (\K\cup J_{1} \cup J_{2}))}\\
	3E(u_{1}, \R^{+})\le&\, E(u,\G\setminus (\K\cup J_{1} \cup J_{2}))\\
	\lim_{x\to 0^+} u_{1}(x)=&\,m.
	\end{split}
	\end{equation*}
	Then taking $u_1$ as above and $u_2\in H^1(\R^+)$ as
	\[	
	u_{2}(x):=\begin{cases}
	\overline{u}(x-L_{1}) & x \in [0, 2L_{1}),\\
	\widetilde{u}(x-2L_{1}) & x\in [2L_{1},2L_{1}+L_{2}),\\
	u_{1}(x-2L_{1}-L_{2}) & \text{elsewhere}
	\end{cases}
	\]
	proves the claim.
\end{proof}

\section*{Acknowledgements}
\noindent The work has been partially supported by the MIUR project ``Dipartimenti di Eccellenza 2018--2022" (CUP E11G18000350001) and by the INdAM GNAMPA project 2020 ``Modelli differenziali alle derivate parziali per fenomeni di interazione".

\end{document}